\documentclass[a4paper,12pt]{amsart}

\usepackage[english]{babel}

\usepackage[utf8]{inputenc}
\usepackage[T1]{fontenc}
\usepackage{lmodern}          
\usepackage{textcomp}
\usepackage[left=3cm, right=3cm, top=3cm]{geometry}
\usepackage{indentfirst}

\usepackage{microtype}

\usepackage{amsmath}
\usepackage{amssymb}         
\usepackage{bbold}
\usepackage{mathtools}
\usepackage{stmaryrd}        
\usepackage{mathrsfs}
\usepackage{dsfont}          
\usepackage[makeroom]{cancel}
\usepackage{thmtools}
\usepackage{graphicx}
\usepackage{tikz}
\usepackage{tikz-cd}
\usetikzlibrary{calc,cd,decorations.pathmorphing,arrows.meta,intersections,spath3,hobby}
\usepackage[all]{xy}
\usepackage{enumitem}

\usepackage[dvipsnames]{xcolor}
\usepackage{cite}
\usepackage[colorinlistoftodos,textsize=tiny]{todonotes}
\usepackage{comment}
\usepackage[normalem]{ulem}

\theoremstyle{plain}
\newtheorem{theorem}{Theorem}[section]
\newtheorem{lemma}[theorem]{Lemma}
\newtheorem*{lemma*}{Lemma}
\newtheorem{proposition}[theorem]{Proposition}

\newtheorem{corollary}[theorem]{Corollary}

\theoremstyle{definition}
\newtheorem{definition}[theorem]{Definition}
\newtheorem{definition-notation}[theorem]{Definition/Notation}
\newtheorem{lemma-definition}[theorem]{Lemma/Definition}

\theoremstyle{remark}
\newtheorem{remark}[theorem]{Remark}
\newtheorem{example}[theorem]{Example}
\newtheorem{notation}[theorem]{Notation}

\definecolor{verylightgrey}{rgb}{0.9,0.9,0.9}

\newcommand\extrafootertext[1]{%
    \bgroup
    \renewcommand\thefootnote{\fnsymbol{footnote}}%
    \renewcommand\thempfootnote{\fnsymbol{mpfootnote}}%
    \footnotetext[0]{#1}%
    \egroup
}

\renewcommand{\phi}{\varphi}

\newcommand{\restr}[2]{\left.\kern-\nulldelimiterspace #1 \right|_{#2}}

\newcommand{\C}{\mathbb{C}}

\newcommand{\Z}{\mathbb{Z}}
\newcommand{\N}{{\mathbb{Z}_{\ge0}}}

\newcommand{\dmg}{\mathcal{Q}}

\newcommand{\id}{\mathds{1}}
\newcommand{\LL}{\mathcal{L}}
\newcommand{\M}{\mathcal{M}}

\definecolor{lightgrey}{rgb}{0.7,0.7,0.7}
\definecolor{ballblue}{rgb}{0.0,0.5,1.0}
\definecolor{purple}{rgb}{0.5,0.0,0.5}
\definecolor{green}{rgb}{0.2,0.6,0.2}

\newcommand{\via}[1]{}

\newcommand{\no}[1]{}

\newcommand{\bmsp}[1][f_-]{\mathrm{BMS}^+(G,\omega;#1)}
\newcommand{\bmsm}[1][f_+]{\mathrm{BMS}^-(G,\omega;#1)}
\newcommand{\BMS}{\mathrm{BMS}(G,\omega)}

\usepackage[hidelinks]{hyperref}
\hypersetup{
    linktoc=all,
}

\numberwithin{equation}{section}
\title{A categorification of Kauffman states for planar graphs}

\author[G. Cerulli Irelli]{Giovanni Cerulli Irelli}
\address{Sapienza Università di Roma; Dipartimento di Scienze di Base e Applicate per l'Ingegneria, Via Antonio Scarpa 14, 00161 Roma, Italy}
\email{giovanni.cerulliirelli@uniroma1.it}

\author[D. Fiorenza]{Domenico Fiorenza}
\address{Sapienza Università di Roma; Dipartimento di Matematica ``Guido Castelnuovo'', P.le Aldo Moro 5, 00185 Roma, Italy}
\email{fiorenza@mat.uniroma1.it}

\author[E. Landi]{Eugenio Landi}
\address{Sapienza Università di Roma; Dipartimento di Matematica ``Guido Castelnuovo'', P.le Aldo Moro 5, 00185 Roma, Italy}
\email{eugenio.landi@uniroma1.it}

\author[M. Matteucci]{Michele Matteucci}
\address{Università di Roma Tre; Dipartimento di Matematica e Fisica, Largo San Leonardo Murialdo 1, 00146 Roma, Italy}
\email{michele.matteucci@uniroma3.it}

\date{\today}

\begin{document}
\begin{abstract}

Given a decorated planar graph $(G,\omega)$, where $G$ is a planar graph and $\omega\in H^1(|\mathcal{Q}G|,\mathbb{Z})$ with $\mathcal{Q}G$ the directed medial graph of $G$, we call some angular functions \emph{$\omega$-compatible} and study two distinct but related directed graphs: $\mathcal{L}(G,\omega)$, which is the directed graph of such functions, and $BMS(G,\omega)$, the directed graph of BMS states which are some pairs of $\omega$-compatible functions plus additional data.
We give sufficient conditions for $\mathcal{L}(G,\omega)$ to be a graded distributive lattice, recovering Kauffman's Clock Theorem when $G$ is a knot diagram.
We also define a potential on $\mathcal{Q} G$ and associate a representation of the corresponding quiver with potential to every BMS state. Under suitable assumptions, this construction yields an isomorphism between $\mathcal{L}(G,\omega)$ and the lattice of subrepresentations of a maximal representation, generalizing a result of Bazier-Matte--Schiffler.
\end{abstract}
\maketitle

\tableofcontents

\section{Introduction}

Given a knot diagram $K$ with a marked edge between two consecutive crossings, one can construct a bipartite graph whose two vertex sets are given by the $n$ crossings of $K$ and the $n$ faces that are not adjacent to the marked edge. A Kauffman state is then a perfect matching of this bipartite graph. In \cite{Kauff-Book}, Kauffman proved that the set of all Kauffman states, ordered by \emph{counterclockwise moves}, is a finite distributive lattice. In the same work, he also showed that the (Conway--)Alexander polynomial can be recovered as a suitable sum indexed by the Kauffman states of $K$.

More recently, Bazier-Matte and Schiffler \cite{BM,BM2,BM3} reinterpreted these constructions in the language of quiver representations and cluster algebras, recovering the Alexander polynomial of a knot as a specialization of the $F$-polynomial of certain indecomposable modules over the Jacobian algebra of a quiver associated with $K$.

The aim of this paper is to extend both pictures from knot diagrams to planar graphs decorated with a cohomology class on their directed medial graph. To such a decorated planar graph $(G,\omega)$ we associate the set $\LL(G,\omega)$ of $\omega$-compatible angular functions and endow it with a directed graph structure via counterclockwise moves along edges of $G$. We then introduce \emph{BMS states}, namely pairs of compatible angular functions together with a dimension vector satisfying an additional compatibility condition, and use them as the bridge between the combinatorics of $\LL(G,\omega)$ and the representation theory of the directed medial graph.

Our first main result shows that nilpotency degree zero (Definition~\ref{def.nilpotency.degree}) is enough to force strong lattice-theoretic properties: if $(G,\omega)$ has nilpotency degree zero, then every connected component of $\LL(G,\omega)$ is the directed graph of a finite graded distributive lattice (Theorem~\ref{thm.connected.components}). A more concrete sufficient condition for connectedness is given in terms of the graph of invisible angular cycles; when this graph is connected, $\LL(G,\omega)$ itself is a distributive lattice (Corollary~\ref{cor-distributive-lattice}). In the case of prime link diagrams endowed with the Kauffman weight, this recovers Kauffman's Clock Theorem.

Our second main result concerns the representation-theoretic side. To every decorated planar graph we associate a quiver with potential and, to every BMS state, a representation of its Jacobian algebra. When the weight is characteristic, we prove that the lattice of plus-subobjects of a non-trivial BMS state is isomorphic to the lattice of subrepresentations of the associated module (Theorem~\ref{thm.plus-subobjects-subrepr}), thus extending the corresponding result of Bazier-Matte--Schiffler \cite[Theorem 6.9]{BM}.

The paper is organized as follows.
In Section~\ref{Sec:PlanarGraph} we introduce decorated planar graphs, angular functions, invisible cycles, and the directed graph $\LL(G,\omega)$.
In Section~\ref{sec-bms-states} we study colorings and BMS states, and prove that each connected component of $\LL(G,\omega)$ is a finite graded distributive lattice under the nilpotency degree zero assumption.
In Section~\ref{sec-kauffman-states} we apply our construction to the case of knot diagrams and recover the theory of Kauffman states and Kauffman's Clock Theorem for prime links.
In Section~\ref{sec-quiver-representations} we associate to $(G,\omega)$ a quiver with potential and define a functor $\mathcal{M}$ taking every BMS state to a representation of its Jacobian algebra.
Finally, in Section~\ref{sec-lattice-isomorphism} we prove that if $\omega$ is a characteristic weight this functor induces, for every BMS state $\xi$, an isomorphism between the lattice of its subobjects and that of the subrepresentations of $\mathcal{M}(\xi)$.

The authors are grateful to Roberto Conti, Francesco Esposito, Daniel Labardini-Fragoso, and Ralf Schiffler for the helpful conversations and suggestions.

This work was supported by the "National
Group for Algebraic and Geometric Structures, and their Applications" (GNSAGA - INdAM) and by NextGenerationEU - PRIN 2022 -B53D23009430006 - 2022S97PMY -PE1-investimento M4.C2.1.1-Structures for Quivers, Algebras and Representations (SQUARE). 

\section{Decorated Planar Graphs}\label{Sec:PlanarGraph}

Let us fix some notation: our graphs $G\coloneqq(V_G,E_G)$ will be, in general, finite, undirected, multigraphs (that is, we admit multiple edges between the same two vertices). We will assume $G$ has no loops, i.e., no edges starting and ending at the same vertex. 
We denote by $E_G(v)$ the set of edges adjacent to $v$, and by $\deg(v)$ the degree of the vertex $v$, that is, the cardinality of $E_G(v)$ counted with multiplicity. We will always assume that $\deg(v)\geq 2$ for every vertex $v$ of $G$. If $V\subseteq V_G$ is a subset of vertices, $E_G(V)$ denotes the set of edges adjacent to at least one vertex $v\in V$.

By a planar graph we mean a graph $G$ embedded in the 2-sphere $S^2$; when we say that a graph is planar, the embedding in the sphere is always understood to be part of the datum. A planar graph has a set of faces $F_G$. We denote by $E_G(f)$ the set of edges adjacent to a face $f$ and, analogously we define $E_G(F)$ when $F\subseteq F_G$. Two half-edges adjacent to a common vertex $v$ and belonging to a common face $f$ determine an angle, that we will always assume to be clockwise oriented in the standard orientation of $S^2$. We call $A_G$ the set of angles of a planar graph. The careful reader can think of angles as adjacent pairs $(v,f)$ provided they keep in mind the many things that can go wrong
(a loop on a vertex $v$ would yield two angles corresponding to the same pair $(v,f)$. From this is easy to give counterexamples even in the loop-free case.)
As above, given a vertex $v$ or a face $f$ the angles adjacent to them will be denoted by $A_G(v), A_G(f)$. Similarly, we define the sets $A_G(V)$ and $A_G(F)$ for any subsets $V\subseteq V_G$ and $F\subseteq F_G$.

We will denote by $\dmg G$ the directed medial graph of $G$.  Here we only recall the material needed for our constructions.
The graph $\dmg G$ is planar. Its vertices correspond to the edges of $G$ and its edges correspond to the angles of $G$ oriented clockwise. We write $\dmg(e)$ and $\dmg(a)$ for the vertex and edge in $\dmg G$ corresponding to the edge $e$ and the angle $a$ in $G$, respectively. For any vertex $v$ of $G$ the edges of $\dmg G$ corresponding to the angles in $A_G(v)$ form a set of edges $\dmg(v)$ of $\dmg G$ bounding clockwise a face in $\dmg G$. For any face $f$ of $G$ the edges of $\dmg G$ corresponding to the angles in $A_G(f)$ form a set of edges $\dmg(f)$ of $\dmg G$ bounding counterclockwise a face in $\dmg G$. Notice that, since we are assuming $\deg(v)\geq 2$  for any vertex $v$ of $G$ and since $G$ contains no loops, the directed graph $\dmg G$ does not contain loops.
The pictorial rule for constructing $\dmg G$ from $G$ is recalled in Figure \ref{fig-directed-medial-graph-construction}.

\begin{figure}[htp]
\centering    \includegraphics[width=236.19562pt,totalheight=65.96858pt]{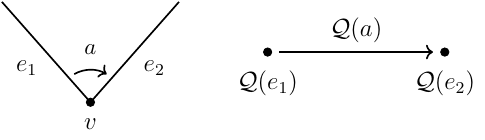}
    \caption{Construction rule of the directed medial graph.}
    \label{fig-directed-medial-graph-construction}
 \end{figure}

Notice that for any edge $e\in E_G$ there are exactly two angles in $A_G$ for which $e$ is a source and exactly two angles for which $e$ is a target. Equivalently, each vertex $\dmg(e)$ in $\dmg G$ is the source of two arrows and the target of other two arrows. In particular, each vertex in
 $\dmg G$ is quadrivalent.

\begin{remark}\label{rem.bounded}
We are not going to make explicit use of this fact, but it is worth noticing that if $G$ is a connected planar graph, then the directed graph $\dmg G$ is strongly connected, i.e., every two vertices $\dmg (e_1)$ and $\dmg (e_2)$ are connected by an oriented path.
\end{remark}

\begin{definition}\label{def.angular.function}
Let $G=(V_G,E_G)$ be a planar graph. An \emph{angular function} on $G$ is a function $g\colon A_G\to \N$.
\end{definition}
Equivalently, an angular function on $G$ is a function $g\colon E_{\dmg G}\to \N$. Let $|\dmg G|$ be the topological realization of the graph $\dmg G$. It has a natural structure of 1-dimensional CW-complex. The orientations on the edges of $\dmg G$ provide a distinguished basis of the free $\Z$-module $C_{1;CW}(|\dmg G|,\Z)$ of 1-dimensional cellular chains of $|\dmg G|$ that, by a little abuse of notation, we will denote by the same symbol $E_{\dmg G}$. Looking at an angular function on $G$ as a function from  $E_{\dmg G}$ to $\Z$, this uniquely extends to a $\Z$-linear map, that we will denote by the same symbol $g$,
\[
g\colon C_{1;CW}(|\dmg G|,\Z)\to \Z,
\]
i.e., it defines a 1-cochain $g\in C^1_{CW}(|\dmg G|,\Z)$. Notice that, since $|\dmg G|$ is 1-skeletal, we have $Z^1_{CW}(|\dmg G|,\Z)=C^1_{CW}(|\dmg G|,\Z)$, i.e., all cellular 1-cochains are 1-cocycles.
\begin{definition}\label{def.decorated-planar-graph}
Let $G=(V_G,E_G)$ be a planar graph, a \emph{weight function} on $G$ is an element $\omega$ in $H^1(|\dmg G|;\Z)$; a \emph{decorated planar graph} is a pair $(G,\omega)$. We say that an \emph{angular function} $g\colon A_G\to \N$ is $\omega$\emph{-compatible} if, when seen as an element in  $Z^1_{CW}(|\dmg G|,\Z)$, it is a representative for the cohomology class $\omega$.
The set of $\omega$-compatible angular functions on a decorated planar graph $(G,\omega)$ will be denoted by the symbol $\LL(G,\omega)$.
\end{definition}

\begin{remark}\label{rem.disjoint.union}
 If $G=G_1\sqcup G_2$, then $H^1(|\dmg G|,\Z)=H^1(|\dmg G_1|,\Z)\oplus H^1(|\dmg G_2|,\Z)$ and $(A_G)^\N=(A_{G_1})^\N\times (A_{G_2})^\N$. This implies $\LL(G,\omega)=\LL(G_1,\omega_1)\times \LL(G_2,\omega_2)$, where $(\omega_1,\omega_2)$ is the image of $\omega$ via the canonical isomorphism $H^1(|\dmg G|,\Z)=H^1(|\dmg G_1|,\Z)\oplus H^1(|\dmg G_2|,\Z)$.
\end{remark}

The $\Z$-module $Z_{1;CW}(|\dmg G|,\Z)$ of cellular 1-cycles of $|\dmg G|$ is a submodule of $C_{1;CW}(|\dmg G|,\Z)$. Since we have fixed the distinguished basis $E_{\dmg G}$ of  $C_{1;CW}(|\dmg G|,\Z)$, we can give the following definition.
\begin{definition}
A 1-cycle $c\in Z_{1;CW}(|\dmg G|,\Z)$ is \emph{non-negative} or an \emph{angular cycle in $G$} if all of its coordinates in the basis $E_{\dmg G}$ of  $C_{1;CW}(|\dmg G|,\Z)$ are non-negative. 
 \end{definition}
Concretely, this means that a non-negative cycle is an oriented cycle in the oriented graph $\dmg G$ such that all of its edges are positively oriented with respect to the orientation of $\dmg G$. Let us write
\[
\lambda_\omega\colon Z_{1;CW}(|\dmg G|,\Z)\to \Z
\]
for the integration map
\[
Z_{1;CW}(|\dmg G|,\Z)= H_1(|\dmg G|,\Z)\xrightarrow{\int_{-}\omega} \Z.
\]
If $g$ is an $\omega$-compatible angular function, for any cycle $c=\sum_{a\in A_G} c_a \dmg(a)$, we have
\begin{equation}\label{eq.cycle}
    \lambda_{\omega} (c)=\displaystyle \sum_{a \in A_G} c_a g(a).
\end{equation}
In particular, for a non-negative cycle all of the coefficients $c_a$ are non-negative and since $g(a)\in \N$, this gives $\lambda_\omega(c)\in \N$ for any non-negative cycle $c$.

\begin{remark}[Weights as functions]\label{rem.omega.as.function}
If $G$ is connected, the choice of an ``external'' face (or face ``at infinity'') for $G$ determines an external face for $\dmg G$. One then has a distinguished basis for $H_{1}(|\dmg G|,\Z)$ given by (the homology classes of) 1-cycles going around the internal faces of $|\dmg G|$. These correspond to vertices and internal faces of $G$. Since $H^1(|\dmg G|,\Z)\cong \mathrm{Hom}_\Z(H_1(|\dmg G|,\Z);\Z)$, we see that, once an external face has been chosen, the datum of a weight function on $G$ is the datum of a function $
\omega\colon V_G\sqcup F^{\mathrm{int}}_G \to \Z$.
Notice that, for an $\omega$-compatible angular function to possibly exist, the weight function $\omega\colon V_G\sqcup F^{\mathrm{int}}_G \to \Z$ must take values in $\N$. Since in what follows we will be only interested in working with $\omega$-compatible angular functions, we will usually think of a weight function on a connected planar graph $G$ as the datum of a choice of an external face of $G$ together with a function
\[
\omega\colon V_G\sqcup F_G^{\mathrm{int}} \to \N.
\]
We can also see $H_1(|\dmg G|,\Z)$ as the free $\Z$-module generated by \emph{all} the faces $\phi_i$ of $\dmg G$, modulo the relation
\[
\phi_{\mathrm{ext}}=\sum_{\phi_i\text{ int}} \phi_i,
\]
where all faces are given clockwise orientations. Recalling that the faces of $\dmg G$ corresponding to vertices in $G$ inherit the clockwise orientation from the orientation of $\dmg G$, while the faces corresponding to faces in $G$ inherit the counterclockwise orientation, and noticing that when realized in the plane by the choice of an external face, the counterclockwise orientation with respect to the external face is seen as the clockwise orientation in the plane, we see that a weight function admitting compatible angular functions is given by a function
\[
\omega\colon V_G\sqcup F_G \to \N
\]
subject to the constraint $
\omega(f_{\mathrm{ext}})=\sum_{v\in V_G} \omega(v)-\sum_{f\in F_G^{\mathrm{int}}}\omega(f)$, i.e., to the constraint
\begin{equation}\label{eq:constraint}
\sum_{v\in V_G} \omega(v)=\sum_{f\in F_G} \omega(f).
\end{equation}
Notice that \eqref{eq:constraint} is independent of the choice of the external face.
With this description, the $\omega$-compatibility of an angular function $g$ becomes purely combinatorial: $g$ is $\omega$-compatible if and only if for any vertex $v$ and every face $f$ of $G$ one has
\begin{equation}\label{eq.omega-and-g}
\omega(v)=\sum_{a\in A_G(v)}g(a); \qquad \omega(f)=\sum_{a\in A_G(f)}g(a).
\end{equation}
We give an example in  Figure \ref{fig-angular-function-example}.
\end{remark}
 \begin{figure}[htp]
\centering    \includegraphics[width=302.74991pt,totalheight=144.64923pt]{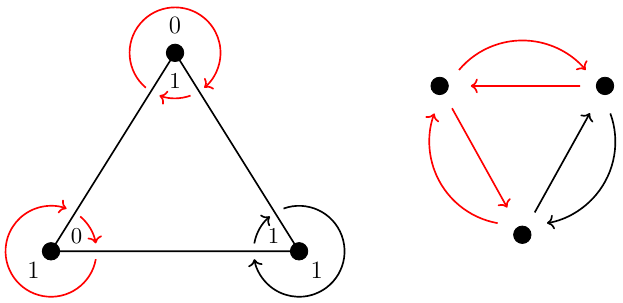}
     \caption{
     On the left: the graph $G$ together with an $\omega$–compatible function $g$, indicated by the black numbers placed near the angles.
     Here the weight function $\omega$ assigns the value $1$ to the top and to the bottom-left vertex, the value $2$ to the bottom-right vertex, and the value $2$ to both the inner and the outer face.
     The orientation of the angles is represented by arrows. On the right: the directed medial graph $\mathcal Q(G)$. The red arrows in the left figure form an angular cycle $c$ of $G$ with $\lambda_{\omega}(c)=2$, whereas the red arrows in the right figure are precisely those corresponding to $\mathcal Q(c)$.}
     \label{fig-angular-function-example}
      \end{figure}
 In view of the application to Kauffman states given in Section~\ref{sec-kauffman-states} it is convenient to give the following definition where we implicitly use Remark~\ref{rem.omega.as.function}.
 \begin{definition}\label{def.char-weight}
    Let $G$ be a connected planar graph. A weight function $\omega:V_G\sqcup F_G\rightarrow\N$ is called a \emph{characteristic weight} if $\textrm{Im}\,\omega\subseteq \{0,1\}$.
 \end{definition}
 \begin{remark}\label{rem.LG.is.finite}
     Equations \eqref{eq.omega-and-g} imply that for any angle $a$ in a connected planar graph $G$ one has $g(a)\leq \max(\omega)$, where $\omega$ is seen as a function $\omega\colon V_G\sqcup F_G^{\mathrm{int}} \to \N$. In particular, this implies that for a connected $G$  the set of $\omega$-compatible functions for a given weight $\omega$ is finite. From Remark \ref{rem.disjoint.union} we obtain that for any planar graph $G$, the set $\LL(G,\omega)$ is finite.
 \end{remark}

By the above considerations, and recalling that the set of angular functions on $G$ is the direct product of the sets of angular functions on the connected components of $G$, we have that $\lambda_\omega$ defines a function
\[
\overrightarrow{\lambda_\omega}\colon \prod_{i\in \pi_0(G)}\{\text{angular cycles in $G_i$}\}\to \mathbb{Z}_{\ge0}^{\pi_0(G)},
\]
where the $G_i$ are the connected components of $G$.
The image of the set of non-zero  angular cycles in $G$ through the map $\overrightarrow{\lambda_\omega}$ is a sublattice  of $\mathbb{Z}_{\ge0}^{\pi_0(G)}$ endowed with the pointwise order. Using this fact we can give the following definition.
\begin{definition}\label{def.nilpotency.degree}
Let $(G,\omega)$ be a decorated planar graph. The \emph{nilpotency degree} of $(G,\omega)$ is the element of $\mathbb{Z}_{\ge0}^{\pi_0(G)}$ defined by
\[
m_{\omega}:=\min \left \{ \overrightarrow{\lambda_\omega}(c) \: ; \: c \text{ a non-zero angular cycle of } G  \right \}.
\]
A non-zero angular cycle $c$ such that $\overrightarrow{\lambda_\omega}(c)=\vec{0}$ is called an \emph{invisible  angular cycle}.
\end{definition}
Clearly, the condition $m_\omega=\vec{0}$ is equivalent to the existence of invisible  angular cycles in every connected component of $G$. In what follows we will focus our attention precisely on decorated planar graphs $(G,\omega)$ with nilpotency degree zero.

\begin{remark}\label{rem.zero}
 It follows from \eqref{eq.cycle} that an $\omega$-compatible angular function $g$ must satisfy $g(a)=0$ for any angle $a$ in an invisible  angular cycle.
\end{remark}
 This leads to the following definition.
 \begin{definition}
 An \emph{invisible angular path} for $(G,\omega)$ is an oriented path in $\dmg G$ such that $g(a)=0$ for any edge $\dmg(a)$ in the path and for any $g\in \LL(G,\omega)$. An invisible angular cycle is then an invisible angular path that happens to be a cycle.
 \end{definition}

With a slight abuse of language, for an $\omega$-compatible angular function $g$ and an invisible angular path $c$ for $(G,\omega)$ we will say that $c$ is an invisible angular path for $g$.
Since the graph $G$ is finite,
we have the following.

\begin{lemma}\label{lemma.exists-e-part1}
    Let $g$ be an $\omega$-compatible angular function for $(G,\omega)$. Assume that for every edge $e$ in $E_G$ we have $g(a)=0$ for at least one of the two incoming angles at $e$ (i.e., for at least one of the two angles for which $e$ is a target). Then for any edge $e$ in $E_G$ there exists an invisible angular path containing $e$ and containing at least one angular cycle (that will then be an invisible angular cycle).
\end{lemma}
\begin{proof}
  Define recursively a sequence of edges in $E_G$, indexed by nonpositive integers, as follows. We set $e_0=e$. Given the $i$-th edge $e_{-i}$ of the sequence, choose one incoming angle $a_{i}$ for $e_{-i}$ such that $g(a_i)=0$, and set $e_{-i-1}$ to be the source of $a_i$. Then for every $n\in \N$ the sequence $(e_{-n},e_{1-n},\dots, e_0)$ is an invisible angular path in $\dmg G$ containing $e$. Since $G$ is finite, for $n$ sufficiently large this path will contain a cycle.
\end{proof}

\medskip

For later use, it is convenient to introduce an auxiliary graph of invisible connected angular cycles.
\begin{definition}\label{Def:GraphInvisible}
Let $(G,\omega)$ be a decorated planar graph. The graph $\Gamma_{(G,\omega)}^{\mathrm{inv}}$ is defined as follows: the vertices of $\Gamma_{(G,\omega)}^{\mathrm{inv}}$ are the invisible connected angular cycles of $(G,\omega)$; there is an edge between $c_1$ and $c_2$ precisely when the angular cycles share a vertex $\dmg(e)$ in $\dmg G$, i.e., when they share a common edge $e$ in $G$. We call $\Gamma_{(G,\omega)}^{\mathrm{inv}}$ the graph of  invisible connected angular cycles of $(G,\omega)$.
\end{definition}

The set $\LL(G,\omega)$ is naturally the set of vertices of a directed graph, whose directed edges are associated with certain combinatorial moves, called \emph{counterclockwise moves along the edges of $G$}, that we now introduce. Recall that an $\omega$-compatible angular function $g$ is in particular a cocycle representative for the cohomology class $\omega$. For any element $\epsilon\in C^0_{CW}(|\dmg G|,\Z)$, i.e., for any function $\epsilon\colon E_G\to \Z$,  the element $g+\delta\epsilon$, where $\delta\colon C^0_{CW}(|\dmg G|,\Z)\to C^1_{CW}(|\dmg G|,\Z)$ is the cellular differential, is again a cocycle representative for $\omega$. In particular, writing $\chi_e$ for the characteristic function of the edge $e$, we have that $g+\delta\chi_e$ is a cocycle representative for $\omega$. The cocycle $g+\delta\chi_e$ is however not necessarily an angular function, since it is not necessarily non-negative on the basis elements of $C^1_{CW}(|\dmg G|,\Z)$ corresponding to angles in $G$. As we are going to show, it is however very simple to describe when this positivity condition is satisfied. The local picture of $G$ around an edge $e$ has the form shown in Figure \ref{fig-edge-adjacent-angles}.
\begin{figure}[htp]
\centering    \includegraphics[width=150.43178pt,totalheight=73.65753pt]{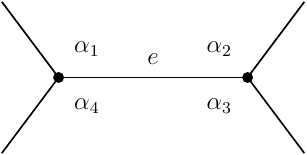}
     \caption{An edge of a planar graph and its four adjacent angles.}
     \label{fig-edge-adjacent-angles}
 \end{figure}
 It is then immediate to see that
 for any angle $\alpha$ in $G$ we have
 \[
 \delta\chi_e(\alpha)=\begin{cases}
     1 &\text{if }\alpha=\alpha_1 \text{ or }\alpha=\alpha_3,\\
     -1 &\text{if }\alpha=\alpha_2 \text{ or }\alpha=\alpha_4,\\
     0 &\text{if }\alpha\not\in \{\alpha_1,\alpha_2,\alpha_3,\alpha_4\}.
     \end{cases}
 \]
\begin{definition}
Let $e$ be an edge of $G$ and let $g$ be an $\omega$-compatible angular function on $G$. We say that $g$ is $e$-movable if $g+\delta\chi_e$ is an angular function, i.e., if it is non-negative on $A_G$. In this case, we write $\mathrm{mov}_e(g)$ for the angular function $g+\delta\chi_e$ and we say that $\mathrm{mov}_e(g)$ is obtained from $g$ through a \emph{counterclockwise move} along the edge $e$.
\end{definition}
\begin{example}
 In the generic situation depicted in Figure \ref{fig-edge-adjacent-angles}, one has that $g$ is $e$-movable if and only if $g(\alpha_2),g(\alpha_4)>0$. When this happens, $\mathrm{mov}_e(g)$ is given by
\[
 \mathrm{mov}_e(g)\colon \alpha\mapsto \begin{cases}
     g(\alpha)+1 &\text{if }\alpha=\alpha_1 \text{ or }\alpha=\alpha_3,\\
     g(\alpha)-1 &\text{if }\alpha=\alpha_2 \text{ or }\alpha=\alpha_4,\\
     g(\alpha) &\text{if }\alpha\not\in \{\alpha_1,\alpha_2,\alpha_3,\alpha_4\}.
     \end{cases}
 \]
 That is, there are two $+1$s in the values of $g$ that move counterclockwise around the two endpoints of $e$. This explains the name counterclockwise move. See Figure \ref{fig-counterclockwise-edge-move}.
\end{example}

\begin{figure}[htp]
\centering    \includegraphics[width=301.01625pt,totalheight=110.93204pt]{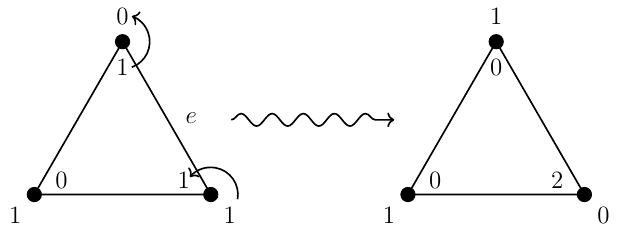}
\caption{A counterclockwise move along the edge $e$ performed on the example in Figure \ref{fig-angular-function-example}.}
     \label{fig-counterclockwise-edge-move}
 \end{figure}

 \begin{remark}
 The angles $\alpha_2,\alpha_4$ in the definition of $e$-movability are those for which the edge $e$ is a source in the directed graph $\dmg G$. Equivalently, they are the two angles that are ``first on the right, second on the left'' for an observer moving from one to the other endpoint of $e$. This is independent of the direction the observer moves along $e$.
 \end{remark}

\begin{definition-notation}\label{def.LG}
 The directed graph $\LL(G,\omega)$ is the directed graph having the set $\LL(G,\omega)$ as set of vertices, and with edges given as follows: there is a directed edge from $g_1$ to $g_2$ for each edge $e$ in $G$ such that $g_1$ is $e$-movable and  $\mathrm{mov}_e(g_1)=g_2$.
 \end{definition-notation}

 \begin{figure}[htp]
\centering    \includegraphics[width=198.20268pt,totalheight=137.38058pt]{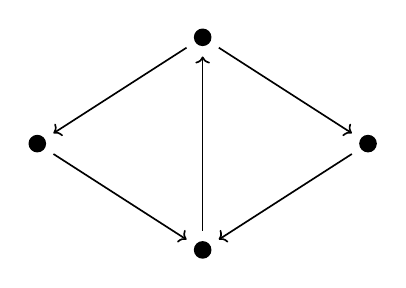}
     \caption{$\LL(G,\omega)$ in the case of Figure \ref{fig-angular-function-example}.}
     \label{fig-l-graph-example}
 \end{figure}

\begin{remark}
If $G=G_1\sqcup G_2$, then $E_G=E_{G_1}\sqcup E_{G_2}$. Via the isomorphism $\LL(G,\omega)=\LL(G_1,\omega_1)\times \LL(G_2,\omega_2)$ from Remark \ref{rem.disjoint.union}, counterclockwise moves associated with edges from $G_i$ only act on the factor $\LL(G_i,\omega_i)$, for $i=1,2$. This means that the isomorphism $\LL(G,\omega)=\LL(G_1,\omega_1)\times \LL(G_2,\omega_2)$ is actually an isomorphism of directed graphs. For this reason, we can restrict our attention to connected graphs, and in what follows we will always assume $(G,\omega)$ denotes a \emph{connected} decorated planar graph.
\end{remark}

\begin{remark}\label{rem.mov-back-and-forth}
Instead of considering $g\mapsto g+\delta\chi_e$ we could have considered $g\mapsto g-\delta\chi_e$. This leads to the notion of anti-$e$-movable angular function and to the clockwise move $\mathrm{mov}_e^{-}$. In particular, an angular function $g$ is anti-$e$-movable if and only if $g(\alpha_1),g(\alpha_3)>0$. If $g$ is $e$-movable, then $\mathrm{mov}_e(g)$ is anti-$e$-movable and $\mathrm{mov}_e^{-}(\mathrm{mov}_e(g))=g$. Similarly, if $g$ is anti-$e$-movable, then $\mathrm{mov}^{-}_e(g)$ is $e$-movable and $\mathrm{mov}_e(\mathrm{mov}^{-}_e(g))=g$.
\end{remark}

We can now state the following theorem, which we will prove in the next section.

\begin{restatable*}{theorem}{structure-of-LG}
    Let $(G,\omega)$ be a decorated planar graph.
If $m_{\omega}=\vec{0}$ then every connected component of $\LL(G,\omega)$ is a finite graded distributive lattice.
\end{restatable*}

\section{Colorings and BMS States}\label{sec-bms-states}
\begin{definition}
    A \emph{dimension vector} on $G$ is a function $d\colon E_G\to \N$.
\end{definition}
We can see a dimension vector $d$ as an element in $C^0_{CW}(|\dmg G|,\Z)$, that we will denote by the same symbol.
\begin{definition}\label{def.colorings}
    Let $(G,\omega)$ be a decorated planar graph and let $f_{+}$ and $f_{-}$ be two $\omega$-compatible angular functions . We call the pair $(f_+,f_-)$ a \emph{coloring} of the decorated graph $(G,\omega)$; $f_+$ is the \emph{plus-function}, whereas $f_-$ is the \emph{minus-function} of the coloring.
\end{definition}
\begin{definition}
Let $(G,\omega)$ be a decorated planar graph. A \emph{weak BMS $\omega$-state}\footnote{After Bazier-Matte--Schiffler.} on $(G,\omega)$ is a triple $(f_+,f_-,d)$, where $(f_+,f_-)$ is a coloring of $(G,\omega)$ and $d$ is a dimension vector on $G$ such that
\begin{equation}\label{eq.compatible}
f_+=f_-+\delta d.
\end{equation}
\end{definition}
\begin{remark}
Equation \eqref{eq.compatible} means that for every angle $a$ in $G$ going from the edge $e_1$ to the edge $e_2$, we have
\begin{equation}\label{eq.bms-state}
        d(e_2) = d(e_1) + f_{+}(a) - f_{-}(a)
    \end{equation}
\end{remark}
In Figure \ref{fig-bms-state-example} we present an example of a BMS state: in each angle $a$ is shown the pair $(f_+(a), f_-(a))$
and near each edge $e$ the value $d(e)$ is specified.

\begin{figure}[htp]
    \centering    \includegraphics[width=191.86992pt,totalheight=147.07463pt]{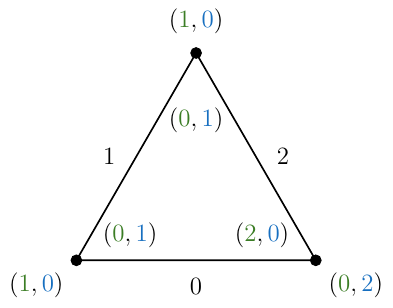}
     \caption{A BMS $\omega$-state on the triangle, where $\omega$ is the weight in the example of Figure \ref{fig-angular-function-example}.}
    \label{fig-bms-state-example}
 \end{figure}

  We have the following refinement of Lemma \ref{lemma.exists-e-part1}.
\begin{lemma}\label{lemma.exists-e-part2}
    Let $(f_+,f_-,d)$ be a weak BMS state for $(G,\omega)$. Assume that for every edge $e$ in $E_G$ with $d(e)>0$ we have $f_+(a)=0$ for at least one of the two incoming angles at $e$. Then for any edge $e$ in $E_G$ with $d(e)>0$ there exists an invisible angular path for $f_+$ containing $e$ and containing at least one (necessarily invisible) angular cycle, such that $d(e')>0$ for every edge $e'$ of the path.
\end{lemma}
\begin{proof}
Reason as in the proof of Lemma \ref{lemma.exists-e-part1}, taking $g=f_+$. Then the only thing left to be shown is that for all the edges $e_{-i}$ in the sequence $(e_{-i})$ we have $d(e_{-i})>0$. This is easily seen inductively: $d(e_0)>0$ by assumption, and, since the angle $a_i$ goes from $e_{-i-1}$ to $e_{-i}$,
\[
d(e_{-i-1})=d(e_{-i})-f_+(a_i)+f_-(a_i)=d(e_{-i})+f_-(a_i)\geq d(e_{-i})>0.
\]
\end{proof}
 \begin{remark}\label{rem.constant.on.edges}
It follows from Remark \ref{rem.zero} that if $c$ is a connected invisible angular cycle for $(G,\omega)$, then $d$ must assign the same value to each edge in $G$ belonging to $c$. Indeed, we will have $f_+(a)=f_-(a)=0$ for any angle in $c$. The conclusion then follows from \eqref{eq.bms-state}. The function $d$ therefore induces a function on the set of vertices of the graph $\Gamma_{(G,\omega)}^{\mathrm {inv}}$ of invisible connected angular cycles of $(G,\omega)$, that is clearly constant on the connected components, and so it induces a function
\[
d_{\mathrm{inv}}\colon \pi_0(\Gamma_{(G,\omega)}^{\mathrm {inv}})\to \N.
\]
\end{remark}
\begin{definition}
     A BMS state for $(G,\omega)$ is a weak BMS state $(f_+,f_-,d)$ such that $d_{\mathrm{inv}}=0$.
     We denote the set of all BMS states for $(G,\omega)$ by the symbol $\BMS$.
\end{definition}
\begin{proposition}\label{prop.weak-is-strong}
    Let $(f_+,f_-,d)$ be a BMS state for $(G,\omega)$. Then there exists an oriented path in $\LL(G,\omega)$ from $f_-$ to $f_+$.
\end{proposition}
\begin{proof}
 We argue by induction on the total dimension $d_{\mathrm{tot}}=\sum_{e\in E_G}d(e)$. If $d_{\mathrm{tot}}=0$ then \eqref{eq.bms-state} gives $f_+=f_-$ and there is nothing to prove. Assume now the statement is true up to $d_{\mathrm{tot}}=n$ and let us prove it for $d_{\mathrm{tot}}=n+1$. There must be at least one edge $e$ in $E_G$ with $d(e)>0$ and such that $f_+(a)>0$ for both incoming angles at $e$. Indeed, if this were not the case, we would be in the hypotheses of Lemma \ref{lemma.exists-e-part2} and there would be an invisible angular cycle $c$ for $(G,\omega)$ with $d_{\mathrm{inv}}(c)>0$, against the assumption that $(f_+,f_-,d)$ is a BMS state. We can then perform the clockwise move at $e$ on $f_+$. Let us set
 \[
 \mathrm{mov}_e^{-}(d)=d-\chi_e.
 \]
 Since $(\mathrm{mov}_e^{-}(d))(e')=d(e')$ for any $e'\neq e$ and $(\mathrm{mov}_e^{-}(d))(e')=d(e)-1\geq 0$, we have that $\mathrm{mov}_e^{-}(d)$ is a dimension vector on $G$. Moreover we have
 \[
 \mathrm{mov}_e^{-}(f_+)=f_+-\delta\chi_e=f_- + \delta(d-\chi_e)=f_-+\delta\mathrm{mov}_e^{-}(d).
 \]
 Hence $(\mathrm{mov}_e^{-}(f_+),f_-,\mathrm{mov}_e^{-}(d))$ is a weak BMS state for $(G,\omega)$. For any invisible angular cycle $c$ of $(G,\omega)$ we have
\[
 0\leq (\mathrm{mov}_e^{-}(d))_{\mathrm{inv}}(c)\leq d_{\mathrm{inv}}(c)=0,
\]
and so $(\mathrm{mov}_e^{-}(f_+),f_-,\mathrm{mov}_e^{-}(d))$ is a BMS state. Finally,
\[
\mathrm{mov}_e^{-}(d)_{\mathrm{tot}}=d_{\mathrm{tot}}-1=n.
\]
Then the induction hypothesis applies and we have an oriented path in $\LL(G,\omega)$ from $f_-$ to $\mathrm{mov}_e^{-}(f_+)$. Since, as noticed in Remark \ref{rem.mov-back-and-forth}, we have $\mathrm{mov}_e(\mathrm{mov}_e^{-}(f_+))=f_+$, we have an oriented path in $\LL(G,\omega)$ from $\mathrm{mov}_e^{-}(f_+)$ to $f_+$. This concludes the proof.
\end{proof}

The proof of Proposition \ref{prop.weak-is-strong} shows we can naturally lift the directed graph structure on $\LL(G,\omega)$ to a directed graph structure on $\BMS$.
\begin{definition}\label{def.moc.bms}
Let $e$ be an edge of $G$ and let $(f_+,f_-,d)$ be a BMS state on $(G,\omega)$. We say that $(f_+,f_-,d)$ is $e$-movable if the angular function $f_+$ is $e$-movable. In this case we write $\mathrm{mov}_e(f_+,f_-,d)$ for the triple
\[
(\mathrm{mov}_e(f_+),f_-,\mathrm{mov}_e(d))=(f_++\delta\chi_e,f_-,d+\chi_e).
\]
\end{definition}
\begin{lemma}\label{lemma.mov.mov}
Assume $(f_+,f_-,d)$ is an $e$-movable BMS state for $(G,\omega)$. Then $\mathrm{mov}_e(f_+,f_-,d)$ is a BMS state for $(G,\omega)$.
\end{lemma}
\begin{proof}
Since $f_+$ is $e$-movable, the pair $(\mathrm{mov}_e(f_+),f_-)$ is a coloring of $(G,\omega)$. Moreover, we have
\[
\mathrm{mov}_e(f_+)=f_++\delta\chi_e=f_-+\delta(d+\chi_e)=f_-+\delta\mathrm{mov}_e(d),
\]
so $\mathrm{mov}_e(f_+,f_-,d)$ is a weak BMS state. Let now $c$ be an invisible angular cycle for $(G,\omega)$. The dimension vector $\mathrm{mov}_e(d)$ is constant on the edges of $c$ by Remark \ref{rem.constant.on.edges}. Since the values of $\mathrm{mov}_e(d)$ coincide with those of $d$ except at the edge $e$, and each angular cycle must contain at least two distinct vertices $\dmg(e_1)$ and $\dmg(e_2)$ since $\dmg G$ has no loops, we see that $(\mathrm{mov}_e(d))(c)=0$. Hence $\mathrm{mov}_e(f_+,f_-,d)$ is a BMS state.
\end{proof}
\begin{remark}
From the proof of Lemma \ref{lemma.mov.mov} we see in particular that if an edge $e$ belongs to an invisible cycle, then $f_+$ (and so any BMS state $(f_+,f_-,d)$) is not $e$-movable.
\end{remark}

\begin{definition-notation}\label{def.bms2}
 The directed graph $\BMS$ is the directed graph having the set $\BMS$ as set of vertices, and with edges given as follows: there is a directed edge from $(f^1_+,f^1_-,d_1)$ to $(f^2_+,f^2_-,d_2)$ for each edge $e$ in $G$ such that $f^1_+$ is $e$-movable and  $\mathrm{mov}_e(f^1_+,f^1_-,d_1)=(f^2_+,f^2_-,d_2)$.
 \end{definition-notation}
 
 \begin{lemma}\label{lem.forget.is.morphism}
 The forgetful map
 \begin{align*}
 \pi\colon \BMS&\to \LL(G,\omega)\\
 (f_+,f_-,d) &\mapsto f_+
 \end{align*}
 is a covering of directed graphs.
 \end{lemma}
 \begin{proof}
 It is immediate from Definitions \ref{def.LG} and \ref{def.bms2} that $\pi$ is a morphism of directed graphs and that it bijectively maps the edges in $\BMS$ stemming from $(f_+,f_-,d)$ to the edges in $\LL(G,\omega)$ stemming from $f_+$. Hence we are only left with showing that $\pi$ is surjective on the set of vertices. But this is immediate: if $g\in \LL(G,\omega)$, then $(g,g,\vec{0})\in \BMS$.
 \end{proof}
\begin{proposition}\label{prop.BMS-is-poset}
The directed graph $\BMS$ is the graph of a poset, i.e., the relation $(f^1_+,f_-^1,d_1)\leq (f_+^2,f_-^2,d_2)$ if and only if there exists a directed path in $\BMS$ going from $(f^1_+,f_-^1,d_1)$ to $(f_+^2,f_-^2,d_2)$ is a partial order relation.
\end{proposition}
\begin{proof}
The relation is clearly reflexive and transitive, so we are left with proving antisymmetry. Let $e_1,\dots, e_n$ be a sequence of edges in $G$ corresponding to a path in $\BMS$ going from $(f^1_+,f_-^1,d_1)$ to $(f_+^2,f_-^2,d_2)$. Then $(d_2)_{\mathrm{tot}}=(d_1)_{\mathrm{tot}}+n$, and so $d_1\leq d_2$ with strict inequality as soon as $n>0$. Therefore, if both $(f^1_+,f_-^1,d_1)\leq (f_+^2,f_-^2,d_2)$ and $(f^2_+,f_-^2,d_2)\leq (f_+^1,f_-^1,d_1)$ hold, then the path between $(f^1_+,f_-^1,d_1)$ and $(f_+^2,f_-^2,d_2)$ contains no edge, and so $(f^1_+,f_-^1,d_1)=(f_+^2,f_-^2,d_2)$.
\end{proof}

In the light of Proposition~\ref{prop.BMS-is-poset} we give the following definition.

\begin{definition}\label{def:plus-subobject}
A BMS state $\xi'$ is called a \emph{plus-subobject} of $\xi$ if $\xi' \le \xi$.
\end{definition}

\begin{remark}
From the proof of Proposition \ref{prop.weak-is-strong} we see that the total dimension $d_{\mathrm{tot}}$ defines a map of posets $d_{\mathrm{tot}}\colon \BMS\to \N$. In other words, $\BMS$ is a graded poset with grading $d_{\mathrm{tot}}$.
\end{remark}

\begin{definition}
For any angular function $g$ in $\LL(G,\omega)$ we denote by $\bmsp[g]$ the subgraph of $\BMS$ whose vertices $(f_+,f_-,d)$ satisfy $f_-=g$.
\end{definition}
\begin{lemma}
The graph of $\bmsp[g]$ is a full subgraph of $\BMS$, i.e., if $\gamma$ is a path in   $\BMS$ joining two vertices in  $\bmsp[g]$, then all vertices in $\gamma$ are in $\bmsp[g]$.
\end{lemma}
\begin{proof}
Immediate: the counterclockwise moves $\mathrm{mov}_e$ do not change the angular function $f_-$.
\end{proof}

Assume now that $(G,\omega)$ has nilpotency degree zero (see Definition \ref{def.nilpotency.degree}), i.e., that every connected component of $G$ contains invisible angular cycles. Under this assumption we have the following.
\begin{lemma}\label{lem.dim-implies-map}
Let $(G,\omega)$ be a decorated planar graph of nilpotency degree zero, let $g$ in $\LL(G,\omega)$ and let $(f^1_+,g,d_1)$ and $(f_+^2,g,d_2)$ be BMS states in $\bmsp[g]$. Then $(f^1_+,g,d_1)\leq (f_+^2,g,d_2)$ if and only if $d_1\leq d_2$ in the pointwise order. In other words the dimension vector
\begin{align*}
\mathrm{dim}\colon \bmsp[g] &\to (E_G)^\N\\
(f_+,g,d)&\mapsto d
\end{align*}
is an embedding of graded posets, where the grading is given by total dimension on both sides.
\end{lemma}
\begin{proof}
The `only if' part is obvious, so we prove the `if' part. Since $d_1\leq d_2$, the function $d_2-d_1\colon E_G\to \Z$ takes its values in $\N$ and so it is a dimension vector. We have
\[
f_+^2=g+\delta d_2=g+\delta d_1+\delta(d_2-d_1)=f_+^1+\delta(d_2-d_1),
\]
so $(f_+^2,f_+^1,d_2-d_1)$ is a weak BMS state for $(G,\omega)$. For every connected invisible cycle $c$ for $(G,\omega)$ we have $d_1(c)=d_2(c)=0$ and so $(d_2-d_1)(c)=0$. Therefore $(f_+^2,f_+^1,d_2-d_1)$ is a BMS state. By Proposition \ref{prop.weak-is-strong}, there is an oriented path in $\LL(G,\omega)$ from $f^1_+$ to $f^2_+$, and so there is an oriented path in $\BMS$ from $(f^1_+,g,d_1)$ to $(f_+^2,g,\tilde{d}_2)$ for a suitable dimension vector $\tilde{d}_2$. We conclude by showing that $\tilde{d}_2=d_2$. We have
\[
\delta(\tilde{d}_2-d_2)=(f_+^2-g)-(f_+^2-g)=0,
\]
hence $\tilde{d}_2-d_2$ is constant on connected components of $G$. Since the nilpotency degree of $(G,\omega)$ is zero, each connected component of $G$ contains at least an invisible angular cycle. On each edge of an invisible angular cycle both $\tilde{d}_2$ and $d_2$ vanish, so in each connected component of $G$ there is at least an edge where $\tilde{d}_2$ and $d_2$ coincide. Therefore, $\tilde{d}_2$ and $d_2$  coincide on every connected component of $G$ and so they coincide everywhere.
\end{proof}

\begin{theorem}\label{thm.bms-lattice}
Let $(G,\omega)$ be a decorated planar graph of nilpotency degree zero, and let $g$ be an $\omega$-compatible angular function.
The map
\[
\mathrm{dim}\colon \bmsp[g]\to (E_G)^\N
\] is an embedding of graded distributive lattices.
\end{theorem}
\begin{proof}
The real content of the statement is the fact that $\bmsp[g]$ is a distributive lattice: all other parts have already been proved. Given two elements $(f_+^1,g,d_1)$ and $(f_+^2,g,d_2)$ in $\bmsp[g]$, let $d^\vee=\max_{\mathrm{pw}}\{d_1,d_2\}$, where $\max_{\mathrm{pw}}$ denotes the pointwise maximum, and let $f_+^\vee=g+\delta d^\vee$. For any angle $a$ in $G$, let $e_1$ and $e_2$ be the source and target edges of $a$, respectively. Then $\delta d^\vee (a)=d^\vee(e_2)-d^\vee(e_1)$. Let $i_0\in\{1,2\}$ be such that $d^\vee(e_1)=d_{i_0}(e_1)$. Then
\begin{align*}
f^\vee_+(a)=g(a)+d^\vee(e_2)-d^\vee(e_1)&=g(a)+d^\vee(e_2)-d_{i_0}(e_1)\\
&\geq g(a)+d_{i_0}(e_2)-d_{i_0}(e_1)=f^{i_0}_+(a)\geq 0,
\end{align*}
so $f^\vee_+$ is an angular function and therefore $(f_+^\vee,g,d^\vee)$ is a weak BMS state. On each edge $e$ of an invisible cycle $c$ we have $d_1(e)=d_2(e)=0$, and so $d^\vee(e)=0$. This implies $d^\vee(c)=0$ for any invisible cycle $c$ for $\omega$, so $(f_+^\vee,g,d^\vee)$ is a BMS state. We define the join of the BMS states $(f_+^1,g,d_1)$ and $(f_+^2,g,d_2)$ as
\[
(f_+^1,g,d_1)\vee (f_+^2,g,d_2)=(f_+^\vee,g,d^\vee).
\]
Since $d^\vee\geq d_1,d_2$, we have from Lemma \ref{lem.dim-implies-map} that   $(f_+^\vee,g,d^\vee)\geq (f_+^1,g,d_1),(f_+^2,g,d_2)$. Assume now $(f_+^3,g,d_3)$ is such that $(f^3_+,g,d_3)\geq (f_+^1,g,d_1),(f_+^2,g,d_2)$. Then $d_3\ge d^\vee$, and so, by  Lemma \ref{lem.dim-implies-map} again,  $(f_+^3,g,d_3)\ge (f_+^\vee,g,d^\vee)$. Similarly, we define the meet
\[
(f_+^1,g,d_1)\wedge (f_+^2,g,d_2)=(f_+^\wedge,g,d^\wedge)
\]
by setting $d^\wedge=\min\{d_1,d_2\}$ and $f^\wedge_+=g+\delta d^\wedge$. To prove the distributivity laws
 \begin{align*}
    x \vee (y \wedge z) &= (x \vee y) \wedge (x \vee z),\\
    x \wedge (y \vee z) &= (x \wedge y) \vee (x \wedge z).
    \end{align*}
    for meets and joins of BMS states, use the injectivity of $\mathrm{dim}\colon \bmsp[g]\to (E_G)^\N$ and the distributivity of $\min$ and $\max$ in $(E_G)^\N$.
\end{proof}

\begin{proposition}\label{prop.bms-finiteness}
Let $(G,\omega)$ be a decorated planar graph of nilpotency degree zero, and let $g$ be an $\omega$-compatible angular function. The distributive lattice $\bmsp[g]$ is finite.
\end{proposition}
\begin{proof}
Consider the forgetful map
\begin{align*}
\pi\colon \bmsp[g]&\to \LL(G,\omega)\\
(f_+,g,d)&\mapsto f_+.
\end{align*}
The set $\LL(G,\omega)$ is finite by Remark \ref{rem.LG.is.finite}. To conclude, we just need to show that the fibers of $\pi$ are finite. Let $h\in \LL(G,\omega)$. If $\pi^{-1}(h)\neq \emptyset$, let $(h,g,d_0)$ be an element in $\pi^{-1}(h)$. If $(h,g,d)$ is another element in the same fiber, then $\delta(d-d_0)=0$ and so $d-d_0$ is constant on connected components of $G$. We reason as in the proof of Lemma \ref{lem.dim-implies-map} to see that $d=d_0$. Hence we see that the fibers of $\pi$ contain at most one element.
\end{proof}

\begin{restatable}{theorem}{structure-of-LG}\label{thm.connected.components}
   Let $(G,\omega)$ be a decorated planar graph of nilpotency degree zero. Then each connected component of $\LL(G,\omega)$ is the directed graph of a graded finite distributive lattice.  More precisely, for any $\omega$-compatible $f$ for $(G,\omega)$ there exists an angular function $f_-$ for $(G,\omega)$ such that the forgetful map
   \begin{align*}
       \pi\colon \bmsp[f_-]&\to \LL(G,\omega) \\
       (f_+,f_-,d)&\mapsto f_+
   \end{align*}
induces an isomorphism of directed graphs $\bmsp[f_-]\to \LL(G,\omega)_f$, where $\LL(G,\omega)_f$ denotes the connected component of $f$ in $\LL(G,\omega)$.
 \end{restatable}
\begin{proof}
If $\LL(G,\omega)$ is empty there is nothing to prove. If $\LL(G,\omega)$ is nonempty, let $h$ be an element in $\LL(G,\omega)$, and let $\bmsm[h]$ be the subgraph of $\BMS$ whose vertices $(h_+,h_-,d)$ satisfy $h_+=h$. By the same argument as in the proof of Proposition \ref{prop.bms-finiteness}, the graph $\bmsm[h]$ is finite. Let
\[
d_{\max,h}=\max_{\mathrm{pw}}\{d\,:\,(h,h_-,d)\in \bmsm[h]\},
\]
and let $f_-=h-\delta d_{\max,h}$. By the same argument as in the proof of Theorem \ref{thm.bms-lattice}, the function $f_-$ is always non-negative, so it is an angular function and $(h,f_-,d_{\max,h})$ is a BMS state. The forgetful map $\pi\colon \bmsp[f_-]\to \LL(G,\omega)$ factors as
\[
\bmsp[f_-]\hookrightarrow \BMS\to \LL(G,\omega),
\]
and so by Lemma \ref{lem.forget.is.morphism} it is a morphism of directed graphs. As seen in the proof of Proposition \ref{prop.bms-finiteness}, the projection map $\pi\colon \bmsp[f_-]\to \LL(G,\omega)$ is injective on the set of vertices. Since, again by Lemma \ref{lem.forget.is.morphism}, $\pi \colon \BMS\to \LL(G,\omega)$ is a covering, this implies that $\pi\colon \bmsp[f_-]\to \LL(G,\omega)$ is an embedding of directed graphs. Since $\bmsp[f_-]$ is a lattice, it is connected as an unoriented graph, and so also its image in $\LL(G,\omega)$ is connected as an unoriented graph. The two BMS states $(h,f_-,d_{\max,h})$ and $(f_-,f_-,\vec{0})$ in $\bmsp[f_-]$ are such that $\vec{0}\leq d_{\max,h}$. Hence, by Lemma \ref{lem.dim-implies-map} we have $(f_-,f_-,\vec{0})\leq (h,f_-,d_{\max,h})$ and so there is an oriented path in $\LL(G,\omega)$ from $f_-$ to $h$. Hence, in particular, the connected component of $h$ and of $f_-$ in $\LL(G,\omega)$ coincide. Since $\pi(f_-,f_-,\vec{0})=f_-$, we have that the image of $\bmsp[f_-]$ in $\LL(G,\omega)$ is contained in the connected component $\LL(G,\omega)_{f_-}$ of $f_-$. We are then left with showing that $\pi$ is surjective on this component. If a directed edge in  $\BMS$ has its source in $\bmsp[f_-]$, then it has also its target in $\bmsp[f_-]$. From this and the fact that $\BMS\to \LL(G,\omega)$ is a covering, we see that in order to prove that $\pi\colon \bmsp[f_-]\to\LL(G,\omega)_{f_-}$ is surjective we only have to show that it is surjective on vertices. Let therefore $k$ be a vertex in $\LL(G,\omega)_{f_-}$. We prove that $k$ is in the image of $\pi$ by induction on the path distance from $k$ to $f_-$. If this distance is zero, then $k=f_-$ and so it trivially is in the image of $\pi$. Assume the statement is true for all $k$ with $\mathrm{dist}(k,f_-)\leq n$ and assume we have $\mathrm{dist}(k,f_-)\leq n+1$. Then we can find $\tilde{k}$ with $\mathrm{dist}(\tilde{k},f_-)\leq n$ and $\mathrm{dist}(\tilde{k},k)\leq 1$. If $\mathrm{dist}(\tilde{k},k)=0$ there's nothing to prove, so let us focus on the case $\mathrm{dist}(\tilde{k},k)=1$. This means that either we have a directed edge in $\LL(G,\omega)$ from $\tilde{k}$ to $k$ or vice versa. The first case is simple: by the inductive assumption, there exists a BMS state $(\tilde{k},f_-,\tilde{d})$ in $\bmsp[f_-]$. Denoting by $e$ the edge of $G$ such that $\mathrm{mov}_e(\tilde{k})=k$, we have
\[
\mathrm{mov}_e(\tilde{k},f_-,\tilde{d})=(k,f_-,\tilde{d}+\chi_e)\in \bmsp[f_-]
\]
and so $k$ is in the image of $\pi$. Assume now we have instead $\mathrm{mov}_e(k)=\tilde{k}$. Then $k-f_-=\delta(\tilde{d}-\chi_e)$, and so $(k,f_-,\tilde{d}-\chi_e)$ is an element of $\bmsp[f_-]$ as soon as $\tilde{d}-\chi_e$ takes non-negative values, i.e., as soon as $\tilde{d}(e)\geq 1$. To see that it is necessarily so, assume $\tilde{d}(e)= 0$. In the same notation as in Figure \ref{fig-edge-adjacent-angles}, denote by $e_i$ the edge distinct from $e$ and adjacent to the angle $\alpha_i$. We have
\begin{align*}
0&=\tilde{d}(e)=\tilde{d}(e_1)+\tilde{k}(\alpha_1)-f_-(\alpha_1)\\
&=\tilde{d}(e_1)+(\mathrm{mov}_e(k))(\alpha_1)-f_-(\alpha_1)\\
&=\tilde{d}(e_1)+k(\alpha_1)+1-f_-(\alpha_1)\\
\end{align*}
and so
\[
f_-(\alpha_1)=\tilde{d}(e_1)+k(\alpha_1)+1\geq 1.
\]
The same argument shows that also $f_-(\alpha_3)\geq 1$. Then $f_-$ is anti-$e$-movable, and we can consider the angular function $\hat{f}_-=\mathrm{mov}_e^-(f_-)$. Let $\hat{d}_{\max,h}=d_{\max,h}+\chi_e$. Then we have
\[
h=f_-+\delta d_{\max,h}=f_-+\delta \hat{d}_{\max,h}-\delta \chi_e=\hat{f}_-+\delta \hat{d}_{\max,h},
\]
so $(h,\hat{f}_-,\hat{d}_{\max,h})$ is an element in $\bmsm[h]$. But then, by definition of $d_{\max,h}$ we have
\[
{d}_{\max,h}(e)+1=\hat{d}_{\max,h}(e)\leq {d}_{\max,h}(e),
\]
which is manifestly absurd.
\end{proof}

Although we are not able to fully characterize the connectedness of $\LL(G,\omega)$ for a decorated planar graph of nilpotency degree zero in terms of cohomological or combinatorial properties of $(G,\omega)$, we are able to provide a simple sufficient condition for connectedness. As we are going to show in Section \ref{sec-kauffman-states}, this is a generalization of Kauffman's Clock Theorem \cite{FNT}.

\begin{theorem}\label{thm-l-graph-connectedness}
    Let $(G,\omega)$ be a connected decorated planar graph of nilpotency degree zero.
    If $\Gamma_{(G,\omega)}^{\mathrm{inv}}$ is connected, then $\LL(G,\omega)$ is connected.
\end{theorem}

\begin{proof}
    Assume that $\LL(G,\omega)_1=\LL(G,\omega)_{f_1}$ and $\LL(G,\omega)_2=\LL(G,\omega)_{f_2}$ are two distinct connected components of $\LL(G,\omega)$. By Theorem \ref{thm.connected.components}, there exist angular functions $f_-^i$ in $\LL(G,\omega)$ such that
    \[
    \pi\colon \bmsp[f_-^i]\xrightarrow{\sim} \LL(G,\omega)_{f_i}
    \]
    for $i=1,2$. Let $f_+^1$ be the maximum in the lattice $\LL(G,\omega)_{f_1}$. The angular function $f_-^2$ is clearly the minimum in the lattice $\LL(G,\omega)_{f_2}$. Since both $f_+^1$ and $f_-^2$ represent the cohomology class $\omega$, there exists $\Delta\colon E_G\to \Z$ such that $f_+^1=f_-^2+\delta\Delta$. By the same argument as in Remark \ref{rem.constant.on.edges}, we have that $\Delta$ is constant on connected 
    invisible angular cycles. Moreover, since $\dmg G$ is connected, the function $\Delta$ is unique up to a constant. Therefore, since $\Gamma_{(G,\omega)}^{\mathrm{inv}}$ is connected, we can fix $\Delta$ to be the unique function $\Delta\colon E_G\to \Z$ with $f_+^1=f_-^2+\delta\Delta$ and such that $\Delta_{\mathrm{inv}}=0$.
    If $\Delta$ were non-negative, then $(f_+^1,f_-^2,\Delta)$ would 
    be a BMS state. But then, by Proposition \ref{prop.weak-is-strong}, $f_+^1$ and $f_-^2$ would be connected by a path in $\LL(G,\omega)$, contradicting the fact that they belong to distinct connected components. Therefore the minimum
    \[
    \mu=\min_{e\in E_G}\Delta(e)
    \]
    is strictly negative. Let $G_\mu\subseteq G$ be the subgraph whose edges are the edges of $G$ such that $\Delta(e)=\mu$. We want to show that for each edge $e$ in $G_\mu$ there exists at least an angle $\alpha$ stemming from $e$ and ending on another edge $e'$ of $G_\mu$. We use notation as in Figure \ref{fig-edge-adjacent-angles} and in the proof of Theorem \ref{thm.connected.components}. By its maximality, the angular function $f_+^1$ is not $e$-movable, so at least one between $f_+^1(\alpha_2)$ and $f_+^1(\alpha_4)$ is zero. Assume we have $f_+^1(\alpha_2)=0$. Then
    \[
    \Delta(e_2)=\Delta(e)-f_-^2(\alpha_2)\geq \Delta(e)=\mu.
    \]
    By definition of $\mu$ this gives $\Delta(e_2)=\mu$, and so $e_2$ is an edge in $G_\mu$. The proof for the case $f_+^1(\alpha_4)=0$ is perfectly analogous. We can therefore form an angular path in $G_\mu$ of arbitrary length. Since $G_\mu$ is finite, this path will necessarily contain angular cycles. By construction, the angular function $f_1^+$ vanishes on all angles in these cycles, so they will be invisible angular cycles for $\omega$. Let us pick one of these cycles $c$, that we may choose to be connected.
    If now $e_0$ is an edge in $c$, we find
    \[
    0=\Delta(e_0)=\mu<0,
    \]
    where the first equality comes from the fact that $\Delta$ is zero on $\pi_0(\Gamma_G^{\mathrm{inv}})$ and so it vanishes on every edge of an invisible connected cycle,
     and the second one from $e_0$ being an edge in $G_\mu$. This contradiction shows that $\LL(G,\omega)$ must be connected.
    \end{proof}
 If $\Gamma^{\mathrm{inv}}_{(G,\omega)}$ is not connected,
 then the connectivity of $\LL(G,\omega)$ is not guaranteed. See
 Figure \ref{fig-hopf-link-example} for an example.

 \begin{figure}[htp]
\centering    \includegraphics[width=220.45248pt,totalheight=162.20651pt]{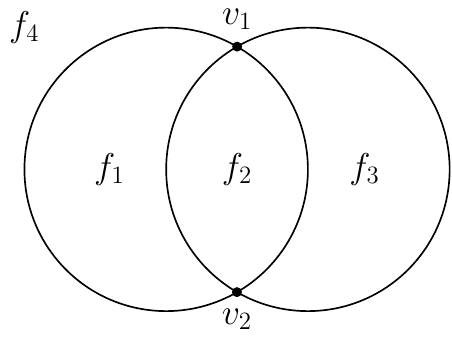}
     \caption{Consider the diagram $G$ of the Hopf link in the figure, and let $\omega$ be the weight defined as follows: $\omega(v_1)=\omega(v_2)=\omega(f_1)=\omega(f_3)=1$ and $\omega(f_2)=\omega(f_4)=0$. Then $\LL(G,\omega)$ consists of two disconnected points.}
     \label{fig-hopf-link-example}
 \end{figure}

Putting Theorems \ref{thm.connected.components} and \ref{thm-l-graph-connectedness} together we obtain the following.
\begin{corollary}\label{cor-distributive-lattice}
    Let $(G,\omega)$ be a connected decorated planar graph of nilpotency degree zero.
    If $\Gamma_{(G,\omega)}^{\mathrm{inv}}$ is connected, then $\LL(G,\omega)$ is a distributive lattice and the minimum $f_-$ of $\LL(G,\omega)$ induces an isomorphism of distributive lattices $\bmsp[f_-]\to \LL(G,\omega)$ given by
    the forgetful map
     \begin{align*}
       \pi\colon \bmsp[f_-]&\to \LL(G,\omega) \\
       (f_+,f_-,d)&\mapsto f_+.
   \end{align*}
\end{corollary}

\section{From Kauffman States to BMS States}\label{sec-kauffman-states}
  We now recall the definition of Kauffman states on a connected link diagram, and show how each Kauffman state can be interpreted as a BMS state. More precisely, we will exhibit a canonical isomorphism of lattices between the lattice of Kauffman states and a lattice of BMS states. Looking at Kauffman states merely as a directed graph, this isomorphism can be seen as a proof of the fact that Kauffman states form a distributive lattice.
  \begin{definition}
   Let $G$ be a connected link diagram, and let $e$ be an edge of $G$. Denote by $f_1$ and $f_2$ the two faces adjacent to $e$.  A \emph{Kauffman state of $G$ relative to the edge $e$} is a subset of angles
\[
K=\{a_1,\dots,a_n\}\subset A_G
\]
such that for every vertex $v\in V_G$ and for every face $f\in F_G-\{f_1,f_2\}$ one has
\[
|A_G(v)\cap K|=|A_G(f)\cap K|=1.
\]
\end{definition}
One customarily depicts a Kauffman state by marking the edge $e$ and by putting a $\bullet$ in the angles in the subset $K$, as in Figure~\ref{fig-kauffman-counterclockwise-move}.
\begin{remark}
The fact that one excludes two faces is motivated by Euler’s formula
\[
|F_G|-|E_G|+|V_G|=2,
\]
that, together with
\[
|E_G|=\displaystyle \frac{1}{2}\sum_{v\in V_G}\deg(v)=2|V_G|,
\]
gives $|V_G|=|F_G|-2$.
\end{remark}

Given a Kauffman state $K$ relative to $e$, a \emph{Kauffman counterclockwise move along the edge $e'$}, with $e'\neq e$, is defined as follows.
Let $v'_1$ and $v'_2$ be the endpoints of $e'$, and let $f'_1$ and $f'_2$ be the two faces adjacent to $e'$.
Assume that $a'_1\in A_G(v'_1)\cap A_G(f'_1)$ and $a'_2\in A_G(v'_2)\cap A_G(f'_2)$ and that, when traversing $e'$ in either direction, the first angle encountered between $a'_1$ and $a'_2$ lies on the right. The move consists of replacing $a'_1$ with the angle that immediately follows it with respect to the counterclockwise cyclic ordering of $A_G(v'_1)$, and replacing $a'_2$ with the angle that immediately follows it with respect to the counterclockwise cyclic ordering of $A_G(v'_2)$. The resulting configuration is again a Kauffman state relative to $e$. See Figure \ref{fig-kauffman-counterclockwise-move} for an example.
\begin{figure}[htp]
\centering    \includegraphics[width=268.74481pt,totalheight=106.54382pt]{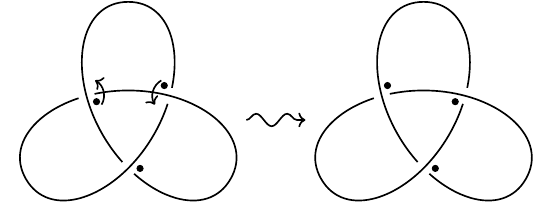}
    \caption{A Kauffman counterclockwise move on the trefoil knot diagram.}
    \label{fig-kauffman-counterclockwise-move}
\end{figure}
Similarly, one defines clockwise moves.
\begin{remark}
    We remark that the Kauffman moves in \cite{FNT} are more general than ours. In this work, as in much of the literature, the moves are along edges, while in \cite{FNT} a move can be done whenever two $\bullet$ are on angles adjacent to edges $e_1,e_2$, possibly equal, themselves adjacent to the same two faces. See \cite{FNT} or \cite[Figure~5]{Kauff-Book}. It is easy to see that for prime knot diagrams the two notions coincide.
\end{remark}
\begin{definition}
We denote by $\mathrm{Kauff}(G,e)$ the directed graph whose vertices are Kauffman states of $G$ relative to $e$, and whose arrows correspond to counterclockwise Kauffman moves.
\end{definition}
\begin{definition}\label{def.kauffman-weight}
Thanks to \eqref{eq:constraint} we can define a characteristic weight $\omega_e$ for $G$ (see Definition \ref{def.char-weight}), depending on the distinguished edge $e$, by
\[
\omega_e(v)=1\quad\text{if }v\in V_G\qquad
\omega_e(f)=
\begin{cases}
1 &\text{if } f\in F_G-\{f_1,f_2\}\\[4pt]
0 &\text{if } f\in\{f_1,f_2\}
\end{cases}
\]
Such a function is called the \emph{Kauffman weight of $G$ relative to the edge $e$}. The fact that $\omega_e$ satisfies equation \eqref{eq:constraint} is an immediate consequence of the equality $|V_G|=|F_G|-2$.

\end{definition}
\begin{figure}[htp]
 \centering    \includegraphics[width=92.02306pt,totalheight=97.35605pt]{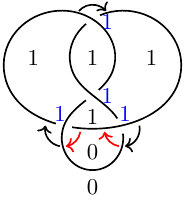}
    \caption{A Kauffman weight on the figure–eight knot:  the two invisible angular cycles are highlighted in red and black respectively.}
    \label{fig-figure-eight-kauffman-weight}
\end{figure}

The following lemma is immediate from equation \eqref{eq.omega-and-g}.
\begin{lemma}
Let $K$ be a Kauffman state for $(G,e)$.
The characteristic function $\chi_K\colon A_G \to \N$ of the subset $K\subseteq A_G$ is an angular function for
$(G,\omega_e)$. This defines an injective map of sets
\begin{align*}
\chi\colon \mathrm{Kauff}(G,e) &\hookrightarrow \LL(G,\omega_e)\\
K &\mapsto \chi_K.
\end{align*}
\end{lemma}

\begin{lemma}
The map $\chi$ is an isomorphism of directed graphs.
\end{lemma}
\begin{proof}
We have already noticed that $\chi$ is injective on the vertices. To see that it is also surjective, notice that equation \eqref{eq.omega-and-g} forces any angular function $f\in \LL(G,\omega_e)$ to be the characteristic function of a subset $K$ of $A_G$, that is, a Kauffman state for $(G,e)$.
The fact that under this bijection the counterclockwise moves on $\mathrm{Kauff}(G,e)$ precisely correspond to counterclockwise moves on $\LL(G,\omega_e)$ is manifest. Conversely, every counterclockwise move in $\LL(G,\omega_e)$ comes from the corresponding counterclockwise Kauffman move on the associated state.
\end{proof}
\begin{lemma}\label{lem-prime-link}
    If $G$ is a prime
    (connected) link diagram, and $e$ is an edge of $G$, then the graph $\Gamma_{(G,\omega_e)}^{\mathrm{inv}}$ of invisible connected angular cycles of $(G,\omega_e)$ is connected.
\end{lemma}
\begin{proof}
    Let $f_1,f_2$ be the faces adjacent to the edge $e$ as above, and let $\omega_e$ be the Kauffman weight relative to $e$.
    Since $\omega_e(f_i)=0$ for $i=1,2$, the angles of $f_i$ form an invisible angular cycle which we call $c_i$; moreover, since $f_1$ and $f_2$ share the adjacent edge $e$, $c_1$ is connected to $c_2$. Suppose there is an invisible angular cycle $c$ which is not connected to $c_1$ nor $c_2$. Suppose $G$ is embedded into a sphere $S$, then $c$ partitions $S$ into two disks, one of which contains neither $f_1$ nor $f_2$: denote this disk by $D_1$ and the other by $D_2$. Let $G_i$ be the subgraph of $G$ determined by the vertices of $G$ contained into $D_i$. Let $f_{\text{ext}}$ be the face of $G_1$ containing $D_2$; let $V_{\text{ext}}$ be the subset of vertices of $G_1$ that are adjacent to $f_{\text{ext}}$, and let $V_{\text{int}}:=V_{G_1}-V_{\text{ext}}$. Clearly $V_{G_2}$ is non-empty. Suppose $V_{G_1}$ is empty, then $G_1$ has a single face $f$, different from $f_1$ and $f_2$, but since $c$ is invisible we must have $\omega (f)=0$, absurd. Therefore $V_{G_1}$ is non-empty. Vertices of $V_{\text{ext}}$ can assume degree $2$, $3$ or $4$ in $G_1$: we denote with $V_{\text{ext}}^{(i)}$ the subset of $V_{\text{ext}}$ of vertices of degree $i$ in $G_1$. Given $v\in V_{\text{ext}}^{(i)}$, the number of edges joining $v$ to some vertex of $G_2$ is $\deg_{G}(v)-\deg_{G_1}(v)=4-i$. Since in any graph $2\left|E\right|=\sum_{v\in V}\deg(v)$, we know that
    \begin{center}
        $2\left |E_{G_1} \right |=2\left |V_{\text{ext}}^{(2)} \right | + 3\left |V_{\text{ext}}^{(3)} \right | + 4\left |V_{\text{ext}}^{(4)} \right |+ 4\left |V_{\text{int}} \right |$.
    \end{center}
    Moreover
    \begin{center}
        $\left | F_{G_1}-\{f_{\text{ext}} \} \right |-\left |E_{G_1}\right |+\left |V_{G_1} \right |=1$.
    \end{center}
    Since $c$ is invisible we must have
    \begin{center}
        $\left | F_{G_1}-\{f_{\text{ext}} \} \right |=\left |V_{G_1}\right |$.
    \end{center}
    From these equalities, we deduce
    \begin{center}
        $2\left |V_{\text{ext}}^{(2)} \right | + \left |V_{\text{ext}}^{(3)} \right |=2$
    \end{center}
     which means $\left |V_{\text{ext}}^{(2)} \right |=1$ and $\left |V_{\text{ext}}^{(3)} \right |=0$ or $\left |V_{\text{ext}}^{(2)} \right |=0$ and $\left |V_{\text{ext}}^{(3)} \right |=2$. This implies that $G$ is a connected sum, absurd.
\end{proof}

\begin{corollary}[Kauffman's Clock Theorem for prime links]
If $G$ is a prime link diagram, then $\mathrm{Kauff}(G,e)$ is a graded distributive lattice.
\end{corollary}
\begin{proof}
   It follows by combining  Lemma~\ref{lem-prime-link} and  Corollary~\ref{cor-distributive-lattice}.
\end{proof}

\section{Quiver representations associated to BMS States}\label{sec-quiver-representations}
In this section we associate to every decorated planar graph $(G,\omega)$ a quiver with potential and a Jacobian algebra, and to every BMS state $\xi$ of $(G,\omega)$
a module $\mathcal{M}(\xi)$ over such Jacobian algebra (Proposition~\ref{prop.M-rep-of-QG}).
We then restrict our attention to the case where $\omega$ is a characteristic weight (Definition \ref{def.char-weight}). In this case, we prove that, for every BMS state $\xi$, the representation $\mathcal{M}(\xi)$ is indecomposable (Theorem~\ref{thm.indecomposability}) and the lattice of its subrepresentations is isomorphic to the lattice formed by the plus-subobjects of $\xi$ (Theorem~\ref{thm.plus-subobjects-subrepr}). If $G$ is the diagram of a prime knot or link and $\omega$ is the Kauffman weight of $G$ and the BMS state is the maximal Kauffman state, then these results are a generalization of \cite[Theorem~6.9]{BM}.

\subsection{Quivers with potential}
We start by recalling basic facts about quivers, quiver representations, quivers with potential, and Jacobian algebras (standard references are \cite{CB,S,ASS,DWZI,DWZII}; see also \cite[Section~2.1]{CLS}). A quiver $Q$ is a quadruple $Q=(Q_0,Q_1,s,t)$ where $Q_0$ is the finite set of its vertices, $Q_1$ the finite set of its arrows and $s,t:Q_1\rightarrow Q_0$ are two functions which provide the orientation of every arrow: $\alpha\in Q_1$ is oriented from its starting vertex $s(\alpha)$ towards its terminal vertex $t(\alpha)$. 

A path of $Q$ is a sequence $\pi=(\alpha_1,\cdots, \alpha_\ell)$ of arrows of $Q$ such that $s(\alpha_{i+1})=t(\alpha_i)$, for all $i=1,\cdots, \ell-1$ and $\ell$ is called the \emph{length} of the path $\pi$. The starting vertex of $\pi$ is $s(\pi):=s(\alpha_1)$ and its terminal vertex is $t(\pi):=t(\alpha_\ell)$, and we use either the notation $\pi=(s(\pi)\mid\alpha_1,\cdots,\alpha_\ell\mid t(\pi))$ or simply $\pi=\alpha_1\cdots\alpha_\ell$. If $s(\alpha_1)=t(\alpha_\ell)$, then we say that $\pi$ is a cyclic path of $Q$. To every vertex $i\in Q_0$ it is associated the lazy path $e_i=(i\mid\mid i)$, i.e. the path starting and ending at vertex $i$ and having length zero. Two paths $\pi=(i\mid \alpha_1,\cdots, \alpha_\ell\mid j)$ and $\pi'=(h\mid \beta_1,\cdots,\beta_{\ell'}\mid k)$ are multiplied as
\[
\pi\pi'=\delta_{j,h}(i\mid \alpha_1,\cdots,\alpha_\ell,\beta_1,\cdots, \beta_{\ell'}\mid k),
\]
where $\delta_{j,h}=1$ if $j=h$, $0$ otherwise.

Following \cite{DWZI}, we now recall the definition of a quiver with potential and its Jacobian algebra. Let $R$ be the commutative $\C$-algebra with basis the lazy paths $e_i$, $i\in Q_0$ with multiplication as above (thus $e_ie_j=\delta_{i,j}e_i$) and let $A$ denote the $R$-bimodule with $\C$-basis the set $Q_1$ of arrows of $Q$ and $R$-bimodule structure $e_iAe_j=\oplus_{\alpha\in Q_1:\, s(\alpha)=i,\,t(\alpha)=j\,}\C \alpha$. The path algebra of $Q$ (see \cite[Def.~2.1]{DWZI}) is the graded tensor algebra $R\langle A\rangle=\bigoplus_{d=0}^\infty A^d$.

For each $i,j\in Q_0$ the component $A^d_{i,j}=e_i A^d e_j$ is the span of all paths of length $d$ starting in $i$ and ending in $j$. We denote by $A^d_{\rm cyc}$ the span of all cyclic paths of length $d$. The complete path algebra of $Q$ (see \cite[Def.~2.2]{DWZI}) is $R\langle\langle A\rangle\rangle=\prod_{d=0}^\infty A^d$, with $\mathfrak{m}$-adic topology given by the two-sided ideal $\mathfrak{m}=\prod_{d=1}^\infty A^d$.
A \emph{potential} on $Q$ is an element of the closed vector subspace $R\langle\langle A\rangle\rangle_{\rm cyc}=\prod_{d=0}^\infty A^d_{\rm cyc}$; thus a potential is a (possibly infinite) linear combination of cyclic paths of $Q$. Given an element $\xi\in A^\ast$ of the linear dual of $A$, the cyclic derivative associated to $\xi$ is the continuous function $\partial_\xi: R\langle\langle A\rangle\rangle_{\rm cyc}\rightarrow R\langle\langle A\rangle\rangle$ acting on cyclic paths by
\[
\partial_\xi(\alpha_1\cdots\alpha_d)=\sum_{k=1}^d \xi(\alpha_k)\alpha_{k+1}\cdots\alpha_d\alpha_1\cdots\alpha_{k-1}
\]
Given a potential $S\in R\langle\langle A\rangle\rangle_{\rm cyc}$, its \emph{Jacobian ideal} $J(S)$ is the closure of the two-sided ideal generated by the elements $\partial_\xi(S)$, for all $\xi\in A^\ast$ and the \emph{Jacobian algebra} of $S$ is the quotient $\mathcal{P}(Q,S)=R\langle\langle A\rangle\rangle/J(S)$. If the quiver $Q$ has no loops, i.e. $A^1_{i,i}=0$ for all $i\in Q_0$, and $S$ is a sum of cyclic paths which are not cyclically equivalent then the pair $(Q,S)$ is called a quiver with potential (see \cite[Definition~4.1]{DWZI}).

\subsection{The quiver with potential of $(G,\omega)$}

We now associate to $(G,\omega)$ a quiver with potential and a Jacobian algebra. The quiver of $(G,\omega)$ is simply $Q=\dmg G$, its directed medial graph. Recall that this quiver has the following properties:  it has no loops, for every vertex $i\in Q_0$ there are exactly two arrows $\alpha_1\neq \alpha_2\in Q_1$  such that $t(\alpha_1)=t(\alpha_2)=i$ and exactly two arrows $\beta_1\neq \beta_2\in Q_1$ such that $s(\beta_1)=s(\beta_2)=i$, $\mid Q_1\mid=2\mid Q_0\mid$, and every two vertices are joined by a path (see Remark~\ref{rem.bounded}).

Recall also that to every vertex $v$ of $G$ is associated a clockwise cycle of $Q$ and it is denoted $\dmg G(v)$; similarly, to every face $f$ of $G$ it is associated a counterclockwise oriented cycle of $Q$ and it is denoted $\dmg G(f)$ \footnote{Strictly speaking $\dmg G(v)$ and $\dmg G(f)$ are elements of the trace space $R\langle\langle A\rangle\rangle/\{R\langle\langle A\rangle\rangle, R\langle\langle A\rangle\rangle\}$ since we have not fixed their initial and terminal vertices. By \cite[Prop.~3.5]{DWZI}, the Jacobian algebra does not depend on this and hence we prefer to be slightly vague here in order to simplify the notation.}.   From now on we use the notation above $R$, $A$, $R\langle A\rangle$ and $R\langle\langle A\rangle\rangle$ for the directed medial graph $Q=\dmg G$.

\begin{notation}
Assume $G$ to be connected. Under this assumption, by Remark \ref{rem.omega.as.function}, a weight function $\omega$ for $G$ can be identified with a function $\omega\colon V_G\sqcup F_G\to \N$. The support of $\omega$, seen as a function on $V_G\sqcup F_G$, is $\mathrm{supp}(\omega)=\{x\in V_G\sqcup F_G\mid\, \omega(x)\neq 0\}$. We denote by $\mathbb{C}[\dmg G]^{\mathrm{cyc}}_{\omega=0}$ the subalgebra of $R\langle\langle A\rangle\rangle_{\rm cyc}$ generated by the elements $\dmg G(x)$ where $x\in V_G\sqcup F_G$ is such that $\omega(x)=0$. We call the elements of $\mathbb{C}[\dmg G]^{\mathrm{cyc}}_{\omega=0}$ the \emph{phantom potentials} for $(G,\omega)$.
\end{notation}

\begin{definition}
Let $(G,\omega)$ be a connected decorated planar graph. We denote by $\xi_\omega$ the least common multiple of the natural numbers $\omega(x)$ with $x$ ranging over the support of $\omega$, and by $p_\omega$ the function $p_\omega\colon \mathrm{supp}(\omega)\rightarrow \N$ given by
\[
p_\omega(x)=\frac{\xi_\omega}{\omega(x)}
\]
The \emph{canonical potential} associated with the weight $\omega$ is the potential
\[
S_{\omega}=  \sum_{v\in V_G\cap \mathrm{supp} (\omega)}\frac{1}{p_{\omega}(v)}\, \dmg G(v)^{p_{\omega}(v)} -\sum_{f\in F_G\cap \mathrm{supp}(\omega)}\frac{1}{p_{\omega}(f)}\, \dmg G(f)^{p_{\omega}(f)}.
\]
A potential $S$ is called $\omega$-\emph{admissible} if $S-S_\omega$ is a phantom potential.
\end{definition}

\begin{example}
If $G$ is a connected planar graph, the potential
\[
S_G=\sum_{v\in V_G}\dmg G(v) - \sum_{f\in F_G}\dmg G(f)
\]
is $\omega$-admissible for every characteristic weight $\omega\colon V_G\sqcup F_G\to \{0,1\}$. Indeed, in this case, $p_\omega(x)=1$ for every $x\in {\rm supp}(\omega)$.
\end{example}

 \subsection{Quiver representations from BMS states} 
 A representation of a quiver $Q=(Q_0,Q_1,s,t)$ is a pair $M=((M_i)_{i\in Q_0}, (M_\alpha)_{\alpha\in Q_1})$ where each $M_i$ is a finite-dimensional $\C$-vector space and each $M_\alpha\colon M_{s(\alpha)}\rightarrow M_{t(\alpha)}$ is a $\C$-linear map. The set of vertices $i$ such that $M_i\neq 0$ is called the \emph{support} of $M$.  The $Q_0$-graded vector space $M=\bigoplus_{i\in Q_0}M_i$ is endowed with a natural structure of $R\langle A\rangle$-module, and we do not distinguish between $Q$-representations and $R\langle A\rangle$-modules (see, for instance, \cite[Theorem~III.1.6]{ASS}). 
 
  Given a path $p=\alpha_1\alpha_2\cdots\alpha_\ell$ of $Q$, we denote by $M_p:=M_{\alpha_\ell}\circ M_{\alpha_{\ell-1}}\circ\cdots\circ M_{\alpha_2}\circ M_{\alpha_1}$ the composition of the corresponding linear maps. For every element $a\in R\langle A\rangle$ of the path algebra, we denote by $M_a$ the corresponding linear combination of the linear maps associated with the paths appearing in $a$.
Given an element of the path algebra $a\in R\langle A\rangle$ we denote by $am\in M$ the image of $m\in M$ under the action of $a$ and denote by $aM$ the vector subspace of all such images.
  A representation $M$ of $Q$ is called nilpotent if there exists $N\gg0$ such that $\mathfrak{m}^N M=0$.
 Equivalently, a representation $M$ of $Q$ is called nilpotent if there exists $N>0$ such that $M_p=0$ for all paths $p$ of length $>N$.
 
 A module over the complete path algebra $R\langle\langle A\rangle\rangle$ is a nilpotent (finite-dimensional) representation of $Q$. A module over the Jacobian algebra $\mathcal{P}(Q,S)$ is a representation $M$ of $R\langle\langle A\rangle\rangle$ such that $\partial_\alpha(S)M=0$ for every $\alpha\in Q_1$. In particular, the simple $R\langle\langle A\rangle\rangle$-modules are one-dimensional and nilpotent. For every vertex $e\in Q_0$ we denote by $S_e$ the simple $R\langle\langle A\rangle\rangle$-module whose support is just the vertex $e$.

 We recall that a morphism between two $Q$-representations $N$ and $M$ is a collection of linear maps $(f_i:N_i\rightarrow M_i)_{i\in Q_0}$ such that $M_\alpha\circ f_{s(\alpha)}=f_{t(\alpha)}\circ N_\alpha$ for every arrow $\alpha\in Q_1$. This notion defines the  abelian category of $Q$-representations, in which surjectivity, injectivity, kernels,  cokernels and direct sum are defined vertex by vertex. In particular, a representation is called \emph{indecomposable} if it is not the direct sum of two non-trivial subrepresentations.

We are now ready to associate to any BMS state a module over the Jacobian algebra $\mathcal{P}(\dmg G, S_\omega)$. We will make use of two linear maps that we call the \emph{plus morphism} and the \emph{minus morphism} which are defined  respectively as follows:
\begin{equation*}
	       \begin{aligned}
	  & \begin{tabular}{cccc}
	    $(+)_n: $& $\mathbb C^n$ & $\longrightarrow$ & $\mathbb C^{n+1}$ \\
	    &$(x_1,\dots,x_n)$&$\mapsto$& $(x_1,\dots,x_n,0)$
	    \end{tabular} \\
	     &\begin{tabular}{cccc}
	    $(-)_n: $& $\mathbb C^n$ & $\longrightarrow$ & $\mathbb C^{n-1}$ \\
	    &$(x_1,\dots,x_n)$&$\mapsto$& $(x_2,\dots,x_n)$
	    \end{tabular}
	    \end{aligned}
\end{equation*}
Taking the standard bases in the domain and codomain, the plus and minus morphisms are respectively represented by the following matrices:
\[
(+)_n = \left ( \begin{array}{c} \id_n\\
      \hline
 \mathbb{0}_{1\times n}
\end{array} \right )
\quad \text{and} \quad
(-)_n = \left ( \begin{array}{c|c}
      \mathbb{0}_{(n-1)\times 1} &
      \id_{n-1}
\end{array} \right ).
\]
where $\id_n$ denotes the $n\times n$ identity matrix and $\mathbb{0}_{m\times n}$ denotes the zero matrix of size $m\times n$.
Since the  dimension of the domain of these two morphisms will always be clear from the context, we will omit the subscripts and simply write $(+)$ instead of $(+)_n$ and $(-)$ instead of $(-)_n$. For example, the morphism $(+)_{n+1} \circ (+)_n$ will simply be written as $(+)^2$.
We formally extend the definition by declaring that $(+)_0$ and $(-)_0$ are zero morphisms and that, for  every $n\ge 1$, $(+)_n^0=(-)_n^0=\id_n$. We denote by $J_n$ (or simply $J$) the nilpotent upper triangular Jordan block of size $n$.

The following useful lemma is straightforward.
\begin{lemma}
For every $n\ge 1$
\begin{equation*}
(+)_{n-1}(-)_{n} =J_n= (-)_{n+1}(+)_n.
\end{equation*}
Given $h,n\ge 1$, for every $c,k\ge 0$, the $h\times n$ matrix representing $(+)^c(-)^k$ is
\begin{equation*}
(+)^c(-)^k=\begin{pmatrix}\mathbb{0}_{(h-c)\times k}&\id_{r}\\\mathbb{0}_{c\times k}&\mathbb{0}_{c\times (n-k)}\end{pmatrix}
\end{equation*}
where $r=\textrm{rk}((+)^c(-)^k)=n-k=h-c$. In particular, for every $m\ge 0$, assuming that $(+)^m(-)^m$ has size $n\times n$, we get
\begin{equation*}
(+)^m(-)^m = J^m=\begin{pmatrix}\mathbb{0}_{(n-m)\times m}&\id_{n-m}\\\mathbb{0}_{m\times m}&\mathbb{0}_{m\times (n-m)}\end{pmatrix}
\end{equation*}
Moreover $(+)_n^c=\begin{pmatrix}\id_{n}\\\mathbb{0}_{c\times n}\end{pmatrix}$ and $(-)_n^k=\begin{pmatrix}\mathbb{0}_{n\times k}&\id_{n}\end{pmatrix}$.

\end{lemma}
An immediate consequence of the first identity above is that any product of $(+)$ and $(-)$ morphisms can be shuffled in such a way that all the $(+)$'s appear on the left and all the $(-)$'s appear on the right.

We are now ready to give the main definition of this section.
\begin{definition}\label{Def:BMSModule}
Given a BMS state $\xi=(f_+,f_-,d)$ for $(G,\omega)$, we denote by
\[
\M(\xi)=((M_e)_{e\in (\dmg G)_0}, (M_\alpha)_{\alpha\in (\dmg G)_1})
\]
the representation of $\dmg G$ defined as follows:
\begin{align}
M_e&=\C^{d(e)},\\
M_\alpha&=(+)^{f_+(\alpha)}(-)^{f_-(\alpha)}:\C^{d(s(\alpha))}\rightarrow \C^{d(t(\alpha))}.
\end{align}
If $f_-(\alpha)\le d(s(\alpha))$ and $f_+(\alpha)\le d(t(\alpha))$, then the BMS relation implies
\[
r_\alpha:=d(s(\alpha))-f_-(\alpha)=d(t(\alpha))-f_+(\alpha)\ge 0,
\]
and we interpret $(+)^{f_+(\alpha)}(-)^{f_-(\alpha)}$ as the matrix
\[
\begin{pmatrix}
\mathbb{0}_{r_\alpha\times f_-(\alpha)} & \id_{r_\alpha}\\
\mathbb{0}_{f_+(\alpha)\times f_-(\alpha)} & \mathbb{0}_{f_+(\alpha)\times r_\alpha}
\end{pmatrix}.
\]
Otherwise, we set $(+)^{f_+(\alpha)}(-)^{f_-(\alpha)}=0$.
\end{definition}

\begin{remark}
    If $G$ is a connected link diagram and $\omega=\omega_e$ is the Kauffman weight of $G$ relative to some edge $e$ (see Definition~\ref{def.kauffman-weight}), then Definition~\ref{Def:BMSModule} becomes the state module defined in \cite[Definition~6.1]{BM}.
\end{remark}

 \begin{lemma}\label{Lem:ShortExactSequence}
  Let $\xi$ and $\xi'$ be two BMS states of $(G,\omega)$  such that $\textrm{mov}_e(\xi')=\xi$ for some vertex $e\in\dmg G_0$. Then there is a short exact sequence of quiver representations:
 \[
 \xymatrix{
 0\ar[r]& M(\xi')\ar^{\iota_{\xi'}}[r]& M(\xi)\ar^\pi[r]& S_{e}\ar[r]& 0
 }
 \]
 \end{lemma}
 \begin{proof}
 Let us use the shorthand $N=M(\xi')$ and $M=M(\xi)$.
 By definition, $N_{e'}=M_{e'}=\C^{d(e')}$ for $e'\neq e$ and $N_e=\C^{d(e)-1}$.
     We define $\iota_{\xi'}$ as follows. For every $e'\in (\dmg G)_0$ we define the linear map $\iota_{\xi'}:N_{e'}\rightarrow M_{e'}$ as follows:
     \[
     \iota_{\xi'}=\left\{
     \begin{array}{cc}
     \id_{d(e')}&\textrm{if }e'\neq e\\
     (+)&\textrm{otherwise}.
     \end{array}
     \right.
     \]
     It is straightforward to check that $M_\alpha\circ\iota_{s(\alpha)}=\iota_{t(\alpha)}\circ N_\alpha$ for every arrow $\alpha$.
 \end{proof}

\begin{lemma}\label{Lem:Nilpotent}
For any BMS state $\xi$, the quiver representation $\M(\xi)$ is nilpotent.
\end{lemma}
\begin{proof}
By repeatedly using Lemma~\ref{Lem:ShortExactSequence} we know that $\M(\xi)$ has a filtration into simple (nilpotent) representations. Thus $\M(\xi)$ itself is nilpotent.
\end{proof}

The following Lemma shows that the opposite implication of Lemma~\ref{Lem:ShortExactSequence} is also true. 
\begin{lemma}\label{Lem:SurjectiveMovable}
    Let $\xi$ be a BMS state for $(G,\omega)$. There exists a non-zero map $M(\xi)\rightarrow S_e$ if and only if $\xi$ is anti-$e$-movable.
\end{lemma}
\begin{proof}
    If $\xi$ is anti-$e$-movable then Lemma~\ref{Lem:ShortExactSequence} shows that there is a non-zero (necessarily surjective) map $M(\xi)\rightarrow S_e$ of quiver representations. 

    On the other hand let us suppose that there exists a vertex $e\in\dmg G_0$ and a non-zero morphism $\xymatrix{f:M(\xi)\ar@{->>}[r]& S_e}$. Let $\alpha_1$ and $\alpha_3$ be the two arrows of $\dmg G$ ending in $e$ (as in Figure~\ref{fig-edge-adjacent-angles}). Let $\xi=(f_+,f_-,d)$. We claim that $f_+(\alpha_1)$ and $f_+(\alpha_3)$ are both positive. Indeed, suppose that $f_+(\alpha_k)=0$ for some $k\in\{1,3\}$. Then $M(\xi)_{\alpha_k}=(-)^{f_-(\alpha_k)}$ is surjective and hence $f_e\circ M(\xi)_{\alpha_k}$ is surjective; in particular, $f_e\circ M(\xi)_{\alpha_k}\neq 0= (S_e)_{\alpha_k}\circ f_{t(\alpha_k)}$. It follows that if $f_+(\alpha_k)=0$ then $f$ is not a morphism of quiver representations, against the hypothesis.
\end{proof}

\begin{proposition}\label{prop.M-rep-of-QG}
Let $(G,\omega)$ be a connected decorated planar graph, and let $S$ be an $\omega$-admissible potential for $G$. Then for every BMS state $\xi=(f_+,f_-,d)$ for $(G,\omega)$ the representation $\M(\xi)$ of $\dmg G$ is a representation of the Jacobian algebra $\mathcal{P}(\dmg G, S)$.
\end{proposition}
\begin{proof}In view of Lemma~\ref{Lem:Nilpotent} we only need to show that for every angle $a\in A_G$ we have $(\partial_a S)\M(\xi)=0$. The angle $a$ belongs exactly to one vertex and one face in $G$ that we denote by $v_a$ and $f_a$, respectively. Since $S=S_\omega+S'$ is the sum of the potential $S_\omega$ and a phantom potential $S'$, by linearity, we get $(\partial_a S)\M(\xi)=(\partial_a S_\omega) \M(\xi)+(\partial_a S')\M(\xi)$. Since $S'$ is  a phantom potential and $(f_+,f_-,d)$ is a BMS state, $f_+(\alpha)=f_-(\alpha)=0$ for every arrow $\alpha$ which lies in a phantom potential. It follows that $(\partial_a S')\M(\xi)=0$ and hence

\begin{align*}
         \partial_{a}S&=\partial_{a}S_\omega=
         \frac{1}{p_\omega(v_a)}\partial_{a}(\dmg G(v_a))^{p_\omega(v_a)}-\frac{1}{p_\omega(f_a)}\partial_{a}(\dmg G(f_a))^{p_\omega(f_a)}\\ \\
         &=
         (\dmg G(v_a))^{p_{\omega}(v_a)-1}\partial_{a}\dmg G(v_a)-(\dmg G(f_a))^{p_{\omega}(f_a)-1}\partial_{a}\dmg G(f_a).
\end{align*}
     Since $\M(\xi)_{\dmg G(v)}=(+)^{\omega(v)}(-)^{\omega(v)}$ and $\M(\xi)_{\dmg G(f)}=(+)^{\omega(f)}(-)^{\omega(f)}$, we have
     \begin{align*}
     \M(\xi)_{\partial_{a}S}&=(+)^{(p_\omega(v_a)-1)\omega(v_a)+\omega(v_a)-f_+(a)}
     (-)^{(p_\omega(v_a)-1)\omega(v_a)+\omega(v_a)-f_-(a)}\\
     &\qquad -
     (+)^{(p_\omega(f_a)-1)\omega(f_a)+\omega(f_a)-f_+(a)}
     (-)^{(p_\omega(f_a)-1)\omega(f_a)+\omega(f_a)-f_-(a)}\\
     &=
     (+)^{p_\omega(v_a)\omega(v_a)-f_+(a)}
     (-)^{p_\omega(v_a)\omega(v_a)-f_-(a)}\\
     &\qquad -
     (+)^{p_\omega(f_a)\omega(f_a)-f_+(a)}
     (-)^{p_\omega(f_a)\omega(f_a)-f_-(a)}\\
     &=(+)^{\xi_\omega-f_+(a)}(-)^{\xi_\omega-f_-(a)}-(+)^{\xi_\omega-f_+(a)}(-)^{\xi_\omega-f_-(a)}=0.
     \end{align*}
\end{proof}

\begin{theorem}\label{thm.indecomposability}
Let $\omega$ be a characteristic weight. If $\xi=(f_+,f_-,d)$ is a non-trivial BMS $\omega$-state, then $\M(\xi)$ is indecomposable if and only if its support is connected.
\end{theorem}

\begin{proof}
Clearly, if the support of $\M(\xi)$ is not connected then $\M(\xi)$ is decomposable. Therefore, we suppose that the support of $\M(\xi)$ is connected. By Fitting’s Lemma, it is necessary and sufficient (since $\mathbb C$ is algebraically closed) to prove that the endomorphism ring of $\M(\xi)$ is a local ring (see e.g. \cite[Corollary~I.4.8]{ASS} or \cite[Section~2]{CB}). Let $F = (F_e)_{e \in E_G} \in \operatorname{End}(\M(\xi))$. For a vertex $e\in \dmg G_0$ let $v(e)$ and $f(e)$ denote the vertex and the face of $G$ adjacent to $e$. Since $\omega$ is a characteristic weight, $\omega(v(e)), \omega(f(e))\in\{0,1\}$ and hence $\M(\xi)_{\dmg G(v(e))}, \M(\xi)_{\dmg G(f(e))}\in \{\id, J_{d(e)}\}$. Since $(f_+,f_-,d)$ is BMS, if either $\M(\xi)_{\dmg G(v(e))}$ or $\M(\xi)_{\dmg G(f(e))}$ is the identity matrix, then $d(e)=0$ and hence $F_e=0$. Let us suppose that $d(e)>0$. Since $F$ is an endomorphism of the quiver representation $\M(\xi)$, the matrix $F_e : \mathbb C^{d(e)} \to \mathbb C^{d(e)}$ must commute with the nilpotent Jordan block $J_{d(e)}$. It is well-known and easy to prove (see e.g.\ \cite[Lemma 3]{CE}) that the only matrices $A$ such that $A J_{d(e)} = J_{d(e)} A$ are the $d(e)\times d(e)$ upper triangular Toeplitz matrices. Recall that given integers $m, n \geq 1$ and $n$ scalars $a_1, \dots, a_n \in k$, the (upper triangular) Toeplitz matrix $\operatorname{Toep}(a_1, \dots, a_n)$ is the $m \times n$ matrix $T = (T_{i,j})$ given by
\[
T_{i,j} =
\begin{cases}
a_{j-i+1} & \text{if } i \leq j \\
0 & \text{if } i > j
\end{cases}
\]

Thus, for every vertex $e\in\dmg G_0$ such that $d(e)>0$ there exist $d(e)$ numbers $t_1(e), \dots, t_{d(e)}(e) \in \mathbb C$ such that $F_e = \operatorname{Toep}(t_1(e), \dots, t_{d(e)}(e))$. We need to prove that $t_1(e) = t_1(e')$ for all $e, e' \in \dmg G_0$ such that $d(e)>0$ and $d(e')>0$. Since the support of $\M(\xi)$ is connected, it suffices to prove the claim only when the two vertices $e,e'$ are joined by an arrow $\alpha: e\rightarrow e'$ in $\dmg G$. The following formulas are straightforward:
\begin{align*}
\operatorname{Toep}(a_1, \dots, a_n)(+) &=
\begin{pmatrix}
\operatorname{Toep}(a_1, \dots, a_{n-1}) \\
\mathbb{0}_{1\times (n-1)}
\end{pmatrix} \\
(+) \operatorname{Toep}(a_1, \dots, a_n) &=
\begin{pmatrix}
\operatorname{Toep}(a_1, \dots, a_n) \\
\mathbb{0}_{1\times n}
\end{pmatrix} \\
\operatorname{Toep}(a_1, \dots, a_n) (-) &=
\begin{pmatrix}
\mathbb{0}_{n\times 1} & \operatorname{Toep}(a_1, \dots, a_n)
\end{pmatrix} \\
(-)\operatorname{Toep}(a_1, \dots, a_n) &=
\begin{pmatrix}
\mathbb{0}_{(n-1)\times 1} & \operatorname{Toep}(a_1, \dots, a_{n-1})
\end{pmatrix}
\end{align*} Since by definition $F_{e'} \circ \M(\xi)_{\alpha} = \M(\xi)_{\alpha} \circ F_e$, the above formulas imply $t_1(e) = t_1(e')$.\\
\end{proof}

\subsection{Isomorphism of lattices}\label{sec-lattice-isomorphism}
 We can now state the last results of the paper. 
 
\begin{theorem}\label{thm.plus-subobjects-subrepr}
Let $\omega$ be a characteristic weight.
If $\xi\in\BMS$, then the lattice of plus-subobjects of $\xi$ is isomorphic to the lattice of subrepresentations of $\M(\xi)$.

\end{theorem}
\begin{proof}
By repeatedly applying Lemma~\ref{Lem:ShortExactSequence} we know that for every plus-subobject $\xi'\le \xi$ there is a monomorphism of quiver representations
\[
\xymatrix{\M(\xi')\ar@{^(->}[r]& \M(\xi)}
\]
To show that every subrepresentation of $\M(\xi)$ is of this form, we proceed by induction on the dimension vector $d$. If $d=0$, then $\xi=(f_-,f_-,0)$ and $\M(\xi)$ is the zero representation; in this case the statement is trivially true. Let $d>0$ and suppose that $U\subset \M(\xi)$ is a proper quiver subrepresentation. Then
$U$ is contained in a maximal subrepresentation $U^{\max}$ of $\M(\xi)$. By maximality, $\M(\xi)/U^{\max}$ is simple and hence there exists a vertex $e\in \dmg G_0$ such that $\M(\xi)/U^{\max}\simeq S_e$. By Lemma~\ref{Lem:SurjectiveMovable} $\xi$ is anti-$e$-movable. Let $\xi'=\mathrm{mov}_e^{-}(\xi)$. We claim that $U^{\max}=\iota_{\xi'}(\M(\xi'))\subset \M(\xi)$, where $\iota_{\xi'}$ is defined in Lemma~\ref{Lem:ShortExactSequence}. If the claim holds, then we are done by induction. Let $K$ be the kernel of a non-zero map $\xymatrix{\M(\xi)\ar@{->>}[r]&S_e}$. Then $\dim(K_i)=d(i)$ if $i\neq e$ and $\dim(K_e)=d(e)-1$. Since $\omega$ is a characteristic weight, the vector subspace $K_e\subset \M(\xi)_e$ is fixed by the nilpotent Jordan block $J_{d(e)}$ and hence $K_e$ is the span of the first $d(e)-1$ elements of the standard basis of $\C^{d(e)}$. This shows that there is a unique (up to scalar multiples) non-zero morphism $\xymatrix{\M(\xi)\ar@{->>}[r]&S_e}$ and hence $K=U^{\max}=\iota_{\xi'}(\M(\xi'))$.
\end{proof}

\begin{corollary}\label{cor.maximal-subrepr}
Let $(G,\omega)$ be a decorated planar graph of nilpotency degree zero, and assume that $\omega$ is a characteristic weight.

Let $C$ be a connected component of $\LL(G,\omega)$, let $f_C^-$ and $f_C^+$ be respectively the minimum and the maximum of the finite distributive lattice $C$,
and let $\xi_C=(f_C^+,f_C^-,d_C)$ be the unique BMS state in $\bmsp[f_C^-]$ mapping to $f_C^+$ under the isomorphism of Theorem~\ref{thm.connected.components}. 
Then $C$ is isomorphic, as a lattice, to the lattice of subrepresentations of $\M(\xi_C)$.
\end{corollary}

\begin{proof}
By Theorem~\ref{thm.connected.components}, the forgetful map induces an isomorphism of distributive lattices
\[
\pi\colon \bmsp[f_C^-]\xrightarrow{\sim} C.
\]
Since $\xi_C$ is the maximum of the lattice $\bmsp[f_C^-]$, the lattice of plus-subobjects of $\xi_C$ coincides with $\bmsp[f_C^-]$. 
Theorem~\ref{thm.plus-subobjects-subrepr} gives an isomorphism between $\bmsp[f_C^-]$ and the lattice of subrepresentations of $\M(\xi_C)$. Composing with $\pi$, we obtain the desired result.
\end{proof}

We can finally state the generalization of \cite[Theorem 6.9]{BM}.

\begin{corollary}
Let $(G,\omega)$ be a decorated planar graph of nilpotency degree zero, assume that $\omega$ is a characteristic weight and that $\Gamma_{(G,\omega)}^{\mathrm{inv}}$ is connected.
Let $f_-$ and $f_+$ be respectively the minimum and the maximum of the distributive lattice $\LL(G,\omega)$, and let $\xi_{\max}=(f_+,f_-,d_{\max})$ be the unique BMS state in $\bmsp[f_-]$ mapping to $f_+$ under the forgetful map. Then $\LL(G,\omega)$ is isomorphic, as a lattice, to the lattice of subrepresentations of $\M(\xi_{\max})$.
\end{corollary}
\begin{proof}
    It follows from Theorem \ref{thm-l-graph-connectedness} and Corollary \ref{cor.maximal-subrepr}.
\end{proof}

\bibliographystyle{alpha}
\bibliography{bibliography}

\end{document}